\documentclass[12pt]{amsart}

\author{Paul \textsc{Poncet}}
\address{CMAP, \'{E}cole Polytechnique, Route de Saclay, 91128 Palaiseau Cedex, France \\
and INRIA, Saclay--\^{I}le-de-France}

\email{poncet@cmap.polytechnique.fr}

\usepackage{stmaryrd}
\usepackage{amssymb}
\usepackage{amsmath}
\usepackage{graphicx}
\usepackage{t1enc}
\usepackage{color}%[dvips]{color}
\usepackage[latin1]{inputenc}
\usepackage[nohug]{diagrams}

\usepackage{yhmath}
\usepackage{fancybox} % pour faire des boîtes sympa (encadrer des théorèmes, etc.)
\usepackage{mathabx}
\usepackage{calrsfs} % nouvelles lettres calligraphiques (commande \mathrsfs)

\usepackage{bbm}
\usepackage{ifpdf}

\def\twoheaduparrow{\rlap{$\uparrow$}\raise.5ex\hbox{$\uparrow$}}%%

\newcommand{\cis}{\Bbbk} % idempotent semifield
\newcommand{\absc}{\dashv}

\newcommand{\sci}{\overline{\mathbb{R}}_+}
\newcommand{\scif}{\mathbb{R}_+}
\newcommand{\scit}{\mathbb{R}^{\max}_+}

\newcommand{\ddint}{\int^{\scriptscriptstyle\infty}\!\!}
\newcommand{\dint}[1]{\int^{\scriptscriptstyle\infty}_{#1}\!\!}

\definecolor{myred}{rgb}{0.75,0,0}
\definecolor{myblue}{rgb}{0,0.44,0.75}
\definecolor{mygreen}{rgb}{0,0.69,0.31}

\numberwithin{equation}{section}

\newtheorem{theorem}{Theorem}[section]
\newtheorem{lemma}[theorem]{Lemma}
\newtheorem{proposition}[theorem]{Proposition}
\newtheorem{corollary}[theorem]{Corollary}

\theoremstyle{definition}
\newtheorem{definition}[theorem]{Definition}

\newtheorem{notations}[theorem]{Notations}
\newtheorem{remark}[theorem]{Remark}
\newtheorem{example}[theorem]{Example}
\newtheorem{translation}[theorem]{Translation}

\newcommand{\Set}{\underline{\mathsf{Set}}}
\newcommand{\Po}{\underline{\mathsf{Po}}}
\newcommand{\F}{\mathsf{F}}
\newcommand{\Filters}{\textcolor{myblue}{\mathsf{Fi}}} % filtres
\newcommand{\PFilters}{\textcolor{mygreen}{\mathsf{PFi}}} % filtres principaux
\newcommand{\UpperStar}{\textcolor{myred}{\mathsf{Up}^*}} % upper subsets
\newcommand{\Z}{\mathsf{Z}}

\begin{document}

%\title{Maxitive integration III \\ Continuous maxitive forms and the Riesz representation theorem} 
%\title{Chapter 3 \\ Continuous linear forms and the idempotent Riesz representation theorem} 
\title{What is the role of continuity \\ in continuous linear forms representation?}

\date{\today}

\subjclass[2010]{06A12, %Semilattices
                 06A15, %Galois correspondences, closure operators 
                 %06A75, %Generalizations of ordered sets
                 06F99, %Ordered structures, none of the above, but in this section
                 12K10, %Semifields
                 18A99, %General theory of categories and functors, none of the above, but in this section
                 28B15} %Set functions, measures and integrals with values in ordered spaces
                 %28C15, %Set functions and measures on topological spaces (regularity of measures, etc.)

\keywords{domains, continuous lattices, $\Z$-theory, max-plus algebra, idempotent analysis, idempotent semifields, idempotent semimodules, Dedekind--MacNeille completion, residuation, maxitive measures, Radon--Nikodym theorem, Riesz representation theorem}

\begin{abstract}
The recent extensions of domain theory have proved particularly efficient to study lattice-valued maxitive measures, when the target lattice is continuous. Maxitive measures are defined analogously to classical measures with the supremum operation in place of the addition. 
Building further on the links between domain theory and idempotent analysis highlighted by Lawson (2004), we investigate the concept of domain-valued %\textit{maxitive} (or \textit{max-linear}) 
\textit{linear forms} on an idempotent (semi)module. %, which we define as a ``point-free'' version of maxitive measures. 
In addition to proving representation theorems for continuous linear forms, we address two applications: the idempotent Radon--Nikodym theorem and the idempotent Riesz representation theorem. %extension problems. 
To unify similar results from different mathematical areas, 
%For unification purposes between similar results arising from different mathematical areas, 
our analysis is carried out in the general $\Z$ framework of domain theory. 
%As an application we give a Riesz representation theorem with respect to the Shilkret (idempotent) integral. 
\end{abstract}

\maketitle

%%%%%%%%%%%%%%%%%%%%%%
%%%%%%%%%%%%%%%%%%%%%%
%%%%%%%%%%%%%%%%%%%%%%
%%%%%%%%%%%%%%%%%%%%%%
\section{Introduction}

Maxitive measures %which are a particular kind of set functions, 
%belong to the scope of interest of idempotent analysis. They 
are defined analogously to classical (additive) measures with the supremum operation $\oplus$ in place of the addition $+$.  
%More precisely, if $\mathrsfs{E}$ is a collection of subsets of some nonempty set $E$ containing the empty set and closed under finite unions, a \textit{maxitive measure} on $\mathrsfs{E}$ is a map $\nu : \mathrsfs{E} \rightarrow \sci$ such that 
%\begin{equation}\label{eqmax}
%\nu(\bigcup_{j\in J} G_j) = \bigoplus_{j \in J} \nu(G_j), 
%\end{equation}
%for every finite  family $\{G_j\}_{j\in J}$ of elements of $\mathrsfs{E}$. 
These measures were first introduced by Shilkret \cite{Shilkret71}, and rediscovered many times. This explains why similar notions and results coexist in the literature, that we tried to survey, unify, and surpass in \cite[Chapter~I]{Poncet11}. % a previous work \cite{Poncet??}. 

%and explored by different communities, in particular by mathematicians involved in capacities and large deviations (e.g.\ Norberg \cite{Norberg86}, O'Brien and Vervaat \cite{OBrien91}, Gerritse \cite{Gerritse96}, Puhalskii \cite{Puhalskii01}), idempotent analysis and max-plus algebra (e.g.\ Maslov \cite{Maslov87}, Bellalouna \cite{Bellalouna92}, Akian et al.\ \cite{Akian94b}, Del Moral and Doisy \cite{delMoral98}, Akian \cite{Akian99}), fuzzy sets (e.g.\ Zadeh \cite{Zadeh78}, Sugeno and Murofushi \cite{Sugeno87}, Pap \cite{Pap95}, de Cooman \cite{DeCooman97}, Nguyen et al.\ \cite{Nguyen03}, Poncet \cite{Poncet07}), or optimisation (e.g.\ Barron et al.\ \cite{Barron00}, Acerbi et al.\  \cite{Acerbi02}), thus many examples of such measures can be found in the literature. For instance, if $E$ is a metric space, the Hausdorff dimension and the Kuratowski measure of non-compacity are maxitive measures on the power set of $E$. See Falconer \cite{Falconer90}, Pap \cite{Pap95}, Nguyen et al.\ \cite{Nguyen03}, Poncet \cite{Poncet??} for further examples. 

%Many other terms exist to express the same (or a similar) concept (see a quite exhaustive list in Poncet \cite{Poncet10d}), sometimes more popular, yet the term ``maxitive'', coined by Shilkret, still has our preference because of its anteriority. 

Maslov's monograph \cite{Maslov87}, in which maxitive measures with values in ordered semirings were considered, testifies to deep %Deep 
connections between idempotent analysis and \textit{order theory} or \textit{lattice theory}. %, since maxitive measures with values in ordered semirings have been considered early (see Maslov \cite{Maslov87}). 
Similar initiatives have been undertaken in the framework of fuzzy set theory, %connections have been developed 
%between fuzzy set theory and order theory, 
where $[0,1]$-valued possibility measures have been replaced by lattice-valued possibility measures (see Greco \cite{Greco87}, Liu and Zhang \cite{Liu94}, de Cooman et al.\ \cite{deCooman01}, Kramosil \cite{Kramosil05a}). 
More recently, the branch of order theory dealing with \textit{continuous lattices} and \textit{domains} turned out to play a crucial role in the study of lattice-valued maxitive measures; see the work of Heckmann and Huth \cite{Heckmann98b, Heckmann98}, treating fuzzy set theory, category theory and continuous lattices, and of Akian \cite{Akian99}, who favoured applications to idempotent analysis and large deviations of random processes. Connections between idempotent mathematics and continuous lattices (or domain theory) also arose in the work of Akian and Singer \cite{Akian03}, and were surveyed by Lawson \cite{Lawson04b}. See also the early developments of Norberg \cite{Norberg89, Norberg97} on domain-valued random variables and the use of continuous (semi)lattices in random set theory. 

The article \cite{Poncet12b} was another contribution to the strengthening of these links; we considered maxitive measures with values in a domain rather than in $\sci$. In the present paper we shall build further on the role of domain theory in idempotent analysis. 
We shall be especially interested in linear forms on a module over an idempotent semifield $\cis$. %, called \textit{linear forms}. 
% as maxitive maps with an additional homogeneity property, since our framework has the particular benefit to encompass maxitive measures and maxitive forms in the same formalism. 

Our motivation partly comes from the following apparent paradox. Let $\nu$ be a completely maxitive measure defined on the open subsets $\mathrsfs{G}(E)$ of a topological space $E$, and taking its values in a complete lattice $\cis$. It is known since Heckmann and Huth \cite{Heckmann98b, Heckmann98} and Akian \cite{Akian99} that, if $\cis$ is a continuous lattice (hence a domain), then $\nu$ admits a cardinal density, i.e.\ is of the form 
\begin{equation}\label{eq:cd}
\nu(\cdot) = \bigoplus_{x \in \cdot} c^{+}(x), 
\end{equation}
for some map $c^{+} : E \rightarrow \cis$. See also \cite[Corollary~II-5.9]{Poncet11}. But Heckmann and Huth proved a stronger result, for they \textit{characterized} continuity of $\cis$ as follows: if $\cis$ is a given complete lattice, then it is continuous \textit{if and only if}, for every topological space $E$, each completely maxitive measure $\nu : \mathrsfs{G}(E) \rightarrow \cis$ admits a cardinal density \cite[Theorem~5]{Heckmann98}. 

Surprisingly, the work of Litvinov et al.\ \cite{Litvinov01} and Cohen et al.\ \cite{Cohen04} seems to contradict this result. 
Indeed, these authors proved a representation theorem for \textit{continuous} linear forms $v$ defined on a complete $\cis$-module $M$, with $\cis$ a complete idempotent semifield: one can write 
\begin{equation}\label{eq:cd2}
v(\cdot) = \langle c, \cdot \rangle, %(\cdot \backslash c)^{-1}, 
\end{equation}
for some $c \in M$, where $\langle c, x \rangle$ denotes a $\cis$-valued operation defined on a subset of $M \times M$. This representation has formal and theoretical affinities with that of Equation~(\ref{eq:cd}). However, its terms require no kind of continuity assumption on $\cis$! Coherently no reference to domain theory appears in the last-mentioned papers. How can one understand this paradox?

%On the one hand, Heckmann and Huth showed the strong result that, given a complete lattice $\cis$, continuity of $\cis$ is equivalent to the fact that every $\cis$-valued maxitive measure $\nu$ defined on the collection of open subsets of any topological space has a cardinal density $c$, i.e.\ is of the form 

%On the other hand, Litvinov et al.\ \cite{Litvinov01} and Cohen et al.\ \cite{Cohen04} proved a representation theorem for continuous linear forms defined on a module over an idempotent semifield $\cis$. This theorem has many formal and theoretical links with the representation~(\ref{eq:cd}). However, its terms require no kind of continuity assumption on $\cis$ (contrarily to what Heckmann and Huth's result suggests), and coherently no reference to domain theory appears in the last-mentioned papers. 

To unravel it, we need to go beyond the tools of classical domain theory, and use instead the general $\Z$ framework of domain theory (see  Bandelt and Ern\'e  \cite{Bandelt83}). 
This is about selecting other subsets than the usual filtered subsets, e.g.\ singletons or nonempty subsets. This is done by a functor $\Z : \Po \rightarrow \Set$ from the category of posets to the category of sets. Then one can redefine the notions of way-above relation and continuous poset. In the case where $\Z$ selects nonempty subsets, a continuous poset is nothing but a \textit{completely distributive} poset or \textit{supercontinuous} poset in the sense of Ern\'e et al.\ \cite{Erne06}. And if $\Z$ selects singletons, it happens that the way-above relation coincides with the partial order $\geqslant$ and that every poset is continuous! This functor is implicitly used in \cite{Litvinov01} and \cite{Cohen04}, and this explains why these articles apparently do not ask for continuity of $\cis$. 
%This possibility of new fruitful links between domain theory and idempotent analysis, emphasizing the works of Litvinov et al.\ \cite{Litvinov01} and Cohen et al.\ \cite{Cohen04}, will be examined in  future work. 

We warn the reader that the following notions will depend on a given functor $\Z$: 
\begin{itemize}
	\item way-above relation, 
	\item continuous idempotent semifield, 
	\item smooth linear map (or form), 
	\item continuous linear map (or form), 
	\item completable and complete modules, 
	\item cuts and normal completion of a completable module, 
	\item strongly archimedean element of a module.% over an idempotent semifield. 
\end{itemize}
%We shall not recall the dependency on $\Z$... 
In works related to $\Z$-theory, it is common practice to constantly recall the dependency on $\Z$ ($\Z$-complete poset, $\Z$-way-above relation, $\Z$-continuous poset, etc.); we believe however that it makes the text heavy and is not really useful if the context is clear. %shall not follow this  However, we consider $\Z$ as a kind of \textit{language} 

A linear form $v : M \rightarrow \cis$ is \textit{smooth} if $v$ commutes with infima of $\Z$-sets, and \textit{continuous} if $v$ is smooth and commutes with arbitrary existing suprema. Under appropriate hypotheses, smooth linear forms can be represented by an ideal of the module $M$; for continuous linear forms, this ideal becomes principal, i.e.\ is generated by an element $c$, and one obtains Equation~(\ref{eq:cd2}) as stated by the following theorem. 

\begin{theorem}\label{thm:litv1} %[\protect{\cite[Theorem~5.2]{Litvinov01}}]\label{thm:litv}
Suppose that $M$ is a complete module over a continuous complete idempotent semifield  $\cis \neq \{ 0, 1\}$, %down-complete idempotent semifield $\cis \neq \{0, 1\}$, 
and let $v : M \rightarrow \cis$. 
%If $\overline{M}$ denotes the normal completion of $M$, 
Then $v$ is a non-degenerate continuous linear form on $M$ if and only if there is an strongly archimedean element $c \in M$ such that $v(\cdot) = \langle c, \cdot \rangle$. 
In this case, $c$ is unique and equals the supremum of the set $\{ 1 \geqslant v \}$. 
\end{theorem}

This result generalizes \cite[Theorems~5.1 and 5.2]{Litvinov01} and \cite[Corollary~39]{Cohen04}. The implicit functor $\Z$ is supposed to be \textit{union-complete}, so that in every continuous poset the way-above relation is interpolating, i.e.\ such that $t \gg r$ implies $t \gg s \gg r$ for some $s$. 

Using Theorem~\ref{thm:litv1} we reprove the idempotent Radon--Nikodym theorem (or Sugeno--Murofushi theorem, see \cite{Sugeno87} and \cite[Chapter~I]{Poncet11}). For this purpose we choose for $\Z$ the functor that selects singletons. 
If $\tau$ (resp.\ $\nu$) denotes the dominating (resp.\ dominated) $\sigma$-maxitive measure, the module we work with is not $L^1_+(\tau)$ but a module $\mathbf{M}$ that depends on both $\tau$ and $\nu$. We show that $\tau$ is localizable (resp.\ $\sigma$-principal) if and only if $\mathbf{M}$ is a complete module (resp.\ a $\sigma$-principal module). Moreover, if $\tau$ is $\sigma$-principal, then every $\sigma$-continuous linear form on $\mathbf{M}$ is continuous. With this result, the idempotent Radon--Nikodym theorem can be deduced easily. 

Unfortunately, Theorem~\ref{thm:litv1} is not sufficient for proving an idempotent version of the Riesz representation theorem. 
The idempotent Riesz theorem usually applies to a linear form $V : M \rightarrow \mathbb{R}_+$ defined on the module $M$ of nonnegative bounded continuous maps of a Tychonoff space. It asserts that $V$ can be expressed as a Shilkret integral with respect to some regular maxitive measure that is finite on compact subsets. Since such a measure always admits a finite cardinal density $c^{+}$, this amounts to writing $V$ as 
$$
V(f) = \bigoplus_{x \in E} \frac{f(x)}{c(x)}, 
$$
for some map $c : E \rightarrow \mathbb{R}_+^*$ (and in fact $c = 1/c^{+}$), where $E$ is the underlying topological space. But $c$ does not need to be continuous, it is only lower-semi\-continuous in general. This means that $c$ is \textit{outside} $M$, a case that is not treated by Theorem~\ref{thm:litv1}. 

To take account of this situation, we introduce \textit{module extensions}, i.e.\ pairs $\overline{M}/M$ with $M$ a submodule of a complete module $\overline{M}$. For instance the module of nonnegative lower-semi\-continuous maps is an extension of the module of nonnegative bounded continuous maps. We obtain the following result. 

\begin{theorem}%\label{supergene}
Suppose that $\overline{M}/M$ is a extension over a complete idempotent semifield $\cis$, and let $v : M \rightarrow \cis$ be a linear form on $M$. Assume that the extension is meet-continuous. Then $v$ is non-degenerate continuous on $\overline{M}/M$ if and only if there is an archimedean element $c$ in $\overline{M}/M$ such that $v(\cdot) = \langle c, \cdot \rangle$. %(\cdot \backslash c)^{-1}$. 
In this case, the supremum of $\{ 1 \geqslant v \}$ in $\overline{M}$ is the least $c$ satisfying $v(\cdot) = \langle c, \cdot \rangle$. %(\cdot \backslash c)^{-1}$. 
\end{theorem}

For simplification purposes this theorem is limited to the case where $\Z$ selects singletons. 
A novel assumption is introduced: we ask for the extension $\overline{M}/M$ to be \textit{meet-continuous}. This specifies that finite infima distribute over directed suprema. 
This result enables one to tackle the idempotent Riesz theorem. We reprove, with a few improvements, a version of this theorem given by Choquet \cite{Choquet54} and proved by Kolokoltsov and Maslov \cite{Kolokoltsov87} in the locally-compact case, and a version due to Breyer and Gulinsky \cite{Breyer96} and also reproved by Puhalskii \cite{Puhalskii01}. We also prove a Riesz like theorem in the case where the topological space $E$ is separable metrizable. 
%Its goal is to develop the theory of poset-valued maxitive measures with the tools of domain theory. We undertake our analysis in the general $\Z$ framework of domain theory (see Bandelt and Ern\'e  \cite{Bandelt83}).

%Section~\ref{secpre} lists the basic definitions and concepts needed in the first part of the paper. 
The paper is organized as follows. 
Section~\ref{secpre2} recalls basics of domains and continuous posets, in the %little-known 
categorical framework of $\Z$-theory. 
Section~\ref{sec:sm} deals with the concepts of idempotent semifields and modules over semirings. 
In Section~\ref{sec:mf} we introduce the notion of linear forms defined on a $\cis$-module, where $\cis$ is an idempotent semifield. We propose a generic way of constructing such maps using ideals of the underlying module. When continuity assumptions on $\cis$ are required, we use the tools of $\Z$-theory introduced in Section~\ref{secpre2}. % and gives characterizations,
%Section~\ref{sec:ext} focuses on extension theorems. 
%Our maxitive maps take values in a partially ordered set which may not be a complete lattice. 
%Section~\ref{sec:conmaxmaps} is dedicated to the comparison between residuated forms and completely maxitive forms, the definitions of which look very close and often coincide. 
In Section~\ref{sec:continuous} our main theorem on representation of continuous linear forms on a complete module is proved. 
In Section~\ref{sec:rn} we go through some applications to maxitive measures and the idempotent Radon--Nikodym theorem. 
Section~\ref{sec:dmcc} provides necessary and sufficient conditions for a module to be embeddable into a complete module. 
In Section~\ref{sec:rmf} we prove a representation theorem for residuated forms on a module extension. % and show a represent resolve the paradox. 
In Section~\ref{sec:riesz} the idempotent Riesz representation theorem is proved.  
%Section~\ref{sec:riesz} establishes different Riesz representation theorems for linear forms. 
%In Section~\ref{sec:ld} we go through some applications of maxitive measures to large deviations. %, 

%In Section~\ref{struc} we look at the structure of the set of linear forms. 
%Section~\ref{optim} focus on the particular case of \textit{optimal measures}. 
%Section~\ref{para: denscard} revisits the problem of finding necessary and sufficient conditions for a maxitive measure $\nu$ to have a \textit{cardinal density}, i.e. a map $c$ such that $\nu(G) = \bigoplus_{x \in G} c(x)$ for all $G$. Some categorical aspects are also considered. 
%Section~\ref{secshi} develops the Shilkret integral and its properties. 
%Section~\ref{secRN} is the heart of this article. Existing Radon-Nikod\'ym theorems for the Shilkret integral are gathered and unified with the help of lattice theory. 
%In Section~\ref{secapp} we go through some applications of maxitive measures to: a) large deviations, emphasizing the work of \cite{Varadhan66}, \cite{Akian99}, \cite{Bell01}, and \cite{Puhalskii01}; b) idempotent probability theory, relying on \cite{Akian95}, \cite{Puhalskii01}, and \cite{Barron03}. 

%%%%%%%%%%%%%%%%%%%%%%%%%%%%%%%%%%%%%%%%%
%%%%%%%%%%%%%%%%%%%%%%%%%%%%%%%%%%%%%%%%%
%%%%%%%%%%%%%%%%%%%%%%%%%%%%%%%%%%%%%%%%%
%%%%%%%%%%%%%%%%%%%%%%%%%%%%%%%%%%%%%%%%%
\section{A primer on $\Z$-theory for continuous posets and domains}\label{secpre2}

A \textit{poset} or \textit{partially ordered set} $(P,\leqslant)$ is a set $P$ equipped with a reflexive, antisymmetric and transitive binary relation $\leqslant$. 
Let us denote by $\Po$ the category of all posets with order-preserving maps as morphisms. A \textit{subset selection} is a function that assigns to each poset $P$ a certain collection $\Z[P]$ of subsets of $P$ called the \textit{$\Z$-sets} of $P$. A \textit{subset system} is a subset selection $\Z$ such that 
\begin{enumerate}
	\item[$i)$] at least one $\Z[P]$ has a nonempty element, 
	\item[$ii)$] for each order-preserving map $f : P \rightarrow Q$, $f(Z) \in \Z[Q]$ for every $Z \in \Z[P]$,
\end{enumerate}
the point $ii)$ meaning that $\Z$ is a covariant functor from $\Po$ to $\Set$ (the category of sets) with $\Z[f]$ defined by $\Z[f](Z) = f(Z)$ if $Z \in \Z[P]$, for every order-preserving map $f : P \rightarrow Q$. 
To this definition, first given by Wright et al.\ \cite{Wright78}, we add a third (unusual but useful in the framework of this paper) condition: 
\begin{enumerate}
	\item[$iii)$] the empty set is not in $\Z[P]$, for all posets $P$. 
\end{enumerate}
%This definition was first given by Wright et al.\ \cite{Wright78}. 
The suggestion of \cite{Wright78} to apply subset systems to the theory of continuous posets was followed by Nelson \cite{Nelson81}, Novak \cite{Novak82b}, Bandelt \cite{Bandelt82}, Bandelt and Ern\'e \cite{Bandelt83}, \cite{Bandelt84}, and this research was carried on by Venugopalan \cite{Venugopalan86}, \cite{Venugopalan88}, Xu \cite{Xu95}, Baranga \cite{Baranga96}, Menon \cite{Menon96}, Shi and Wang \cite{Shi96}, Ern\'e \cite{Erne99}, \cite{Erne01} among others. %See \cite{Poncet08} for a survey. 
Conditions $i)$ and $ii)$ together ensure that each $\Z[P]$ contains all singletons. 
%In \cite{Nelson81} and \cite{Baranga96}, condition $i)$ is remplaced by
%\begin{enumerate}
%	\item[$i')$] for some poset $P$, $\Z[P]$ contains a set with at least two elements,
%\end{enumerate}
%in order to avoid the degenerate subset system for which the nonempty sets of $\Z[P]$ are the singleton sets only, for every poset $P$. This ensures that every \textit{$\Z$-continuous function} $f : P \rightarrow Q$ (i.e. which commutes with existing sups of $\Z$-sets, that is such that $\bigoplus f(Z)$ exists for every $Z \in \Z[P]$ admitting a sup and $\bigoplus f(Z) = f(\bigoplus Z)$) is order-preserving. 

The basic example of subset system is the set of directed subsets of $P$. This subset system is behind the classical theory of continuous posets and domains, see the monograph by Gierz et al.\ \cite{Gierz03}. 	
Here are some further examples:
\begin{enumerate}
%	\item Taking $\Z[P]$ to be the set of 
	\item \textcolor{myred}{Taking $\Z[P]$ as the set of all nonempty subsets of $P$ works well for investigating completely distributive lattices, see Ern\'e et al.\  \cite{Erne06}. Completely distributive lattices were initially examined by Raney \cite{Raney52}, \cite{Raney53}.} 
	\item \textcolor{myblue}{The case where $\Z[P]$ is the set of filtered subsets of $P$ was used for instance by G.\ Gerritse \cite{Gerritse97}, Jonasson \cite{Jonasson98}, Akian and Singer \cite{Akian03}. See also \cite{Poncet12b}. } 
	\item \textcolor{mygreen}{If $\Z[P]$ is the set of all singletons of $P$, then $\Z$ is also a subset selection. } % the way-above relation $y \gg x$ (defined below) reduces to the partial order $y \geqslant x$.} %And yet it will play a  crucial in Chapter~\ref{ch:maxmes}. 
	\item A series of papers deals with the case where $\Z[P]$ is the set of chains of $P$, see Markowsky and Rosen \cite{Markowsky76b}, and Markowsky \cite{Markowsky77}, \cite{Markowsky81a}, \cite{Markowsky81b}. Using the Hausdorff maximality theorem, relations between directed subsets and chains %and the derived notions of continuity 
	were explored by Iwamura \cite{Iwamura44}, Bruns \cite{Bruns67}, and Markowsky \cite{Markowsky76}. See also Ern\'e \cite[p.\ 54]{Erne99}. 
	\item The case where $\Z[P]$ is the set of nonempty finite subsets of $P$ was investigated by Martinez \cite{Martinez72}. %, and is linked with the abstract convexity theory developed in van de Vel's monograph \cite{vanDeVel93}. 
	See also Frink \cite{Frink54} and Ern\'e \cite{Erne81}. 
%	\item For $\Z[P]$ the set of Frink ideals (resp. Cauchy ideals) of $P$, see
%	\item If $\Z[P]$ is the set of countably generated ideals, one gets the object studied in \cite{Banaschewski81}.  (In this paper $P$ is lattice with top and bottom such that every countable set has a join.)
\end{enumerate}

Rather than $\Z$, we shall often deal with the subset selection $\F$, %(or $\uparrow\!\! \Z$), 
defined by $\F[P] = \{ \uparrow\!\! Z : Z \in \Z[P] \}$, where $\uparrow\!\! Z$ is the upper subset generated by $Z$, i.e.\ $\uparrow\!\! Z := \{ y \in P : \exists x\in Z, x\leqslant y \}$. The elements of $\F[P]$ are the \textit{$\F$-sets}, or the ($\Z$-)\textit{filters}, of $P$. Although $\F$ is not a subset system in general, it satisfies the following conditions:
\begin{enumerate}
	\item[$i)$] at least one $\F[P]$ has a nonempty element, 
	\item[$ii')$] for each order-preserving map $f : P \rightarrow Q$, $\uparrow\!\! f(F) \in \F[Q]$ for every $F \in \F[P]$, 
	\item[$iii)$] an $\F$-set is never empty. %the empty set is not in $\F[P]$, for all posets $P$.  
\end{enumerate}
A subset selection $\F$ derived from a subset system $\Z$ as above 
%or satisfying properties $i)$ and $ii')$ 
will be called a \textit{filter selection}. Note that, like $\Z$, $\F$ is functorial, i.e.\ $\F[g \circ f] = \F[g] \circ \F[f]$ for all order-preserving maps $f : P \rightarrow Q$ and $g : Q \rightarrow R$, if one naturally defines $\F[f](F) = \uparrow\!\! f(F)$ for all $F \in \F[P]$. 

\begin{translation}[Filter selections]
The first three examples of subset systems given above lead to the following filter selections, respectively:
\begin{enumerate}
%	\item $\F[P]$ is the set of \textit{ideals} (in the sense of \cite{Gierz03}) of $P$ with supremum,
	\item \textcolor{myred}{$\F[P]$ is the set $\UpperStar[P]$ of \textit{nonempty upper subsets} of $P$, } 
	\item \textcolor{myblue}{$\F[P]$ is the set $\Filters[P]$ of \textit{filters} (in the sense of \cite{Gierz03}) of $P$, } 
%	\item $\Z[P]$ is the set of finite subsets of $P$,
	\item \textcolor{mygreen}{$\F[P]$ is the set $\PFilters[P]$ of \textit{principal filters} of $P$. } 
\end{enumerate}
\end{translation}

%A \textit{$\Z$-ideal} (or an \textit{$\I$-set}) of $P$ is a lower set generated by some $\Z$-set. The set of $\Z$-ideals of $P$ is denoted by $\I[P] = \I_{\Z}[P] := \{ \downarrow Z : Z \in \Z[P] \}$. If $\Z$ is a subset system, all principal ideals are $\Z$-ideals, hence the functor $\I : \Po \rightarrow \Set$ (with $\I[f](I) := \downarrow f(I)$ for any order-preserving map $f: P \rightarrow Q$ and any $I \in \I[P]$) is a \textit{standard extension} in the sense of \cite{Erne91}, or a \textit{subset system} in the sense of \cite{Novak82b}.

%A poset $P$ is \textit{$\Z$-complete} if every $\Z$-set has an infimum in $P$. 
We now introduce the \textit{way-above relation}, which in our context is more relevant than the usual \textit{way-below relation}. Thus, our notions of continuous posets and domains are dual to the traditional definitions. The way-above relation has already been used to study lattice-valued upper-semi\-continuous functions, see for instance \cite{Gerritse97} and \cite{Jonasson98}; see also \cite{Poncet12b}. 
We say that $y \in P$ is \textit{way-above} $x\in P$, written $y \gg x$, if, for every $\F$-set $F$ with infimum, $x \geqslant \bigwedge F$ implies $y \in F$. %$y \geqslant z$, for some $z \in \Z$. 
We use the notations $\twoheaduparrow x = \{ y \in P : y \gg x \}$, and for $A \subset P$, $\twoheaduparrow A = \{ y \in P : \exists x \in A, y \gg x \}$. 
The poset $P$ is \textit{continuous} if every element is the $\F$ infimum of elements way-above it, i.e.\ $\twoheaduparrow x \in \F[P]$ and $x = \bigwedge \twoheaduparrow x$ for all $x \in P$. A \textit{domain} is a continuous poset in which every $\F$-set has an infimum. %These definitions agree with \cite{Wright78}, \cite{Novak82b}, \cite{Bandelt83}, \cite{Shi96}. 

\begin{translation}[Continuous posets]
For our three examples of subset systems, the notion of continuous posets translates respectively as follows:
\begin{enumerate}
%	\item $\F[P]$ is the set of \textit{ideals} (in the sense of \cite{Gierz03}) of $P$ with supremum,
	\item \textcolor{myred}{if $\F = \UpperStar$, then a poset is continuous if and only if it is completely distributive (complete distributivity is sometimes called \textit{supercontinuity}), } 
	\item \textcolor{myblue}{if $\F = \Filters$, then a poset is continuous if and only if it is continuous in the sense of \cite{Poncet12b},} 
%	\item $\Z[P]$ is the set of finite subsets of $P$,
	\item \textcolor{mygreen}{if $\F = \PFilters$, then the way-above relation $y \gg x$ reduces to the partial order $y \geqslant x$, and every poset is continuous. } 
\end{enumerate}
\end{translation}

For a poset $P$, the way-above relation is \textit{additive} if, for all $x \in P$, the subset $\{ y \in P : x \gg y\}$ is either empty or directed, i.e.\ if whenever $x \gg y$ and $x \gg y'$, we have $x \gg z$ for some $z \in P$ such that $z \geqslant y$ and $z \geqslant y'$. A continuous poset with an additive way-above relation is \textit{stably-continuous}. 
\textcolor{mygreen}{With respect to the filter selection $\PFilters$, every poset is stably-continuous. }

A poset $P$ has the \textit{interpolation property} if, for all $x, y \in P$ with $y \gg x$, there exists some $z \in P$ such that $y \gg z  \gg x$. 
For continuous posets in the classical sense, it is well known that the interpolation property holds, see e.g.\ \cite[Theorem~I-1.9]{Gierz03}. This is a crucial feature that is behind many important results of the theory. 
%In the classical theory of continuous posets, a crucial feature is the interpolation property. 
For an arbitrary choice of $\Z$, however, this needs no longer to be true. %(see the counterexample below). %, although no counterexample appears in the literature. 
Deriving sufficient conditions on $\Z$ to recover the interpolation property is the goal of the following theorem. 
%$P$ is \textit{strongly continuous} if it is continuous and $\gg$ has the \textit{interpolation property}, which means that for each $x, y \in P$, if $y \gg x$, then there exists some $z \in P$ such that $y \gg z  \gg x$. 
%In \cite{Novak82a} a $\Z$-algebraic poset is what is called here a $\Z$-inductive poset.
The subset selection $\F$ is \textit{union-complete} if, for every $V \in \F[\F[P]]$ (where $\F[P]$ is considered as a poset ordered by reverse inclusion $\supset$), $\bigcup V \in \F[P]$. As explained in \cite{Erne99}, this condition embodies the fact that finite unions of finite sets are finite, $\supset$-filtered unions of filtered sets are filtered, etc. 
The following theorem restates a result due to \cite{Novak82b} and \cite{Bandelt83} in its dual form. We give the proof here for the sake of completeness. 
%, where union-completeness for a subset selection $\Z$ is defined by the less general condition $V \in \Z[\uparrow\!\!\Z[P]] \Rightarrow \bigcup V \in \uparrow\!\!\Z[P]$.  % ??, where it is assumed that $\Z$ is a union-complete subset system (this condition implies that $\F := \uparrow\!\! \Z$ is union-complete, while the converse is untrue). 

\begin{theorem}\cite{Novak82b, Bandelt83}
If $\F$ is a union-complete filter selection, then every continuous poset has the interpolation property. 
\end{theorem}

\begin{remark}
In the context of $\Z$-theory, many authors (see \cite{Novak82b}, \cite{Bandelt83}, \cite{Venugopalan86}) call \textit{strongly continuous} a continuous poset with the interpolation property. % is called \textit{strongly continuous} in \cite{Novak82}, \cite{Bandelt83}, . 
\end{remark}

\begin{proof}
Let $P$ be a continuous poset, and let $x \in P$. We need to show that $F \subset \twoheaduparrow F$, where $F$ denotes the $\F$-set $F = \twoheaduparrow x$. For this purpose we first prove that $\twoheaduparrow F$ is an $\F$-set. Write $\twoheaduparrow F = \bigcup_{y \in F} \twoheaduparrow y = \bigcup V$, where $V$ is the collection of subsets contained in some $\twoheaduparrow y$, $y \in F$. Considering the order-preserving map $f : P \ni y \mapsto \twoheaduparrow y \in \F[P]$ (recall that $\F[P]$ is ordered by reverse inclusion) and using Property~$ii')$ above, we have $V  = \uparrow\!\!f(F) \in \F[\F[P]]$. Since $\F$ is union-complete, one has $\twoheaduparrow F = \bigcup V \in \F[P]$. 
Since $P$ is continuous, 
$$
x = \bigwedge \twoheaduparrow x = \bigwedge F = \bigwedge_{y \in F} y = \bigwedge_{y \in F} (\bigwedge \twoheaduparrow y) = \bigwedge (\bigcup_{y \in F} \twoheaduparrow y) = \bigwedge \twoheaduparrow F. 
$$ 
The definition of the way-above relation and the fact that $\twoheaduparrow F \in \F[P]$ give $y \in \twoheaduparrow F = \twoheaduparrow (\twoheaduparrow x)$, for all $y \in \twoheaduparrow x$.  This proves that $P$ has the interpolation property.
\end{proof}

All subset systems mentioned above are union-complete. It remains an open problem to exhibit a continuous poset with respect to some subset system that does not satisfy the interpolation property. 
%As this does not appear in the literature, we provide an example of a continuous poset which does not satisfy the interpolation property. Consider the subset system $\Z$ that selects subsets with at most two elements. Clearly, $\F = \uparrow\!\! \Z$ is not union-complete. Now take as the poset $E$ the (dual) infinite complete binary tree. Then $E$ is continuous, for each element is the infimum of the two elements right above it, and both of them are way-above $x$. But $E$ fails to have the interpolation property, for it would imply the existence of a \textit{compact} element, i.e.\ of some $x \in E$ such that $x \gg x$, which is not possible.  

We should stress the fact that the machinery of category theory is justified as long as relations between posets are examined. If a single poset $P$ is at stake, having just a collection of subsets of $P$ at disposal could be sufficient, as in the works \cite{Bandelt82}, \cite{Bandelt84}, \cite{Xu95} (where the letter $\mathfrak{M}$ is used for the collection of selected subsets). In the present work, we hope that the relevance of using functorial (filter) selections will be made clear.

\section{Semirings, semifields, modules over a semiring}\label{sec:sm}

%In this section we introduce some elements from order theory. This will help us to better understand what is at work behind Radon--Nikodym properties. Here we derive a representation theorem of linear forms on idempotent forms, originally due to Cohen et al.\ \cite{Cohen04}. In Section~\ref{treillisRN} we shall apply these results to $\sigma$-maxitive measures. 

\subsection{Semirings, semifields}

%A \textit{semilattice} 
A \textit{semiring} is an abelian monoid $(\cis, \oplus, 0)$ endowed with an additional binary relation $\times$ (the multiplication) that is associative, has a unit $1 \neq 0$, distributes over $\oplus$, and admits $0$ as absorbing element. 
A semiring is \textit{idempotent} (or is a \textit{dioid}, see Baccelli et al.\ \cite{Baccelli92} or Gondran and Minoux \cite{Gondran08}) if $\oplus$ is idempotent, i.e.\ $t \oplus t = t$ for all $t$, and \textit{commutative} if the multiplication is commutative. 
An (\textit{idempotent}) \textit{semifield} is an (idempotent) semiring in which every non-zero element has a multiplicative inverse. We do not assume a semifield to be commutative in general (see however Remark~\ref{rk:iwa}). 
Notice that, if $\cis$ is an idempotent semifield, then $\cis \setminus \{0\}$ is a \textit{lattice-group}. This implies that $\cis$ is a distributive lattice; in particular, every nonempty finite subset of $\cis$ has an infimum, and we have 
\begin{equation}\label{eq:cisinf}
s \wedge t = (s^{-1} \oplus t^{-1})^{-1}, 
\end{equation}
for all $s, t \in \cis\setminus\{0\}$. 

A dioid has a natural structure of partially ordered set with $s \leqslant t \Leftrightarrow  s \oplus t = t$, whose bottom element is $0$. With this point of view $s \oplus t$ is nothing but the supremum of $\{s, t\}$, hence every dioid is a commutative idempotent monoid, i.e.\ a \textit{semilattice}. %For this reason we may preferably write $x \oplus y$ instead of $x + y$. 
Given a filter selection $\F$, a dioid is \textit{complete} %\textit{locally-complete} (resp.\ \textit{complete}) 
if every upper-bounded %(resp.\ arbitrary) 
subset $T$ has a supremum such that  
%it is locally-complete (resp.\ complete) as an ordered set and 
\begin{equation}\label{eq:infsup}
s (\bigoplus T) = \bigoplus_{t \in T} s t, \qquad (\bigoplus T) s = \bigoplus_{t \in T} t s, 
\end{equation}
for all $s$, and if every $\F$-set $F$ has an infimum such that 
\begin{equation}\label{eq:infsupf}
s (\bigwedge F) = \bigwedge_{f \in F} s f, \qquad (\bigwedge F) s = \bigwedge_{f \in F} f s, 
\end{equation}
for all $s$. 
\textcolor{mygreen}{With respect to the filter selection $\PFilters$ that selects principal ideals, Equations~(\ref{eq:infsupf}) are trivial. } %; moreover, a complete dioid is a \textit{quantale with unit} in the sense of Rosenthal \cite{Rosenthal90}. }

In an idempotent semifield, Equations~(\ref{eq:infsup}) and (\ref{eq:infsupf}) are satisfied for all subsets $T$ (resp.\ $F$) with supremum (resp.\ with infimum). Thus, an idempotent semifield is complete if and only if every upper-bounded subset has a supremum. This makes the notion of complete idempotent semifield independent of the filter selection $\F$. 

The following lemma, which will be used many times in this paper, mimics a result by Akian and Singer \cite[Lemma~2.1]{Akian03}. 
%, in which case Equations~(\ref{eq:infsup}) are automatically satisfied, every nonempty subset has an infimum, and  
%$$
%s . \bigwedge T = \bigwedge_{t \in T} s.t, \qquad (\bigwedge T).s = \bigwedge_{t \in T} t.s, 
%$$
%for all $s$ and all nonempty subsets $T$. 

\begin{lemma}[Extends \protect{\cite[Lemma~2.1]{Akian03}}]\label{lem:akiansinger}
Let $\F$ be a filter selection, and let $\cis$ be an idempotent semifield. For all $r, s, t \in \cis$ with $r \neq 0$, $t \gg s$ implies $t r \gg s r$. 
\end{lemma}

\begin{proof}
Let $F$ be an $\F$-set of $\cis$ with infimum such that $s r \geqslant \bigwedge F$. The map $f : \cis \rightarrow \cis$ defined by $f(u) = u r^{-1}$ is  order-preserving, hence the set $\uparrow\!\! f(F) = f(F)$ is an $\F$-set. Since $t \gg s \geqslant (\bigwedge F) r^{-1} = \bigwedge f(F)$, we have $t \in f(F)$, so $t r \in F$. %there is some $u \in F$ such that $t \geqslant f(u)$, i.e\ $t r \geqslant u$. 
This shows that $t r \gg s r$. 
\end{proof}

From now on we use the acronym \textit{cis} for a complete idempotent semifield distinct from $\{0, 1\}$. A cis is never a complete lattice; if it were, there would be a greatest element $\top$, and we would have $\top \geqslant 1 \Rightarrow \top^2 \geqslant \top \Rightarrow \top^2 = \top \Rightarrow \top = 1$, while $\top = 1$ is only possible if the cis coincides with $\{0, 1\}$ (a case that is excluded in the definition of a cis). % a contradiction. 

\begin{remark}\label{rk:iwa}
It is worth recalling that, by the Iwasawa theorem, every cis is commutative (see e.g.\ Birkhoff \cite[Theorem~28]{Birkhoff67}). %Wehrung \cite[Th\'eor\`eme~8.3]{Wehrung95} and \cite[Corollaire~8.11]{Wehrung95}). 
\end{remark}

\begin{remark}[On quasifields]
Litvinov et al.\ \cite{Litvinov01} defined a \textit{quasifield} as a dioid in which every non-zero element is the supremum of invertible elements and such that $t \leqslant 1$ whenever the subset $\{ t^n : n = 1, 2, \ldots \}$ is upper-bounded. They showed that every quasifield distinct from $\{0,1\}$ can be embedded into a cis, and asserted that, conversely, every cis is a quasifield. This latter point indeed holds, for if, for some $t \neq 0$, the subset $\{ t^n : n = 1, 2, \ldots \}$ is upper-bounded, and if $s$ is its supremum, then with Equations~(\ref{eq:infsup}) we have $t s \leqslant s$; since $s \neq 0$ we deduce that $t \leqslant 1$. 
\end{remark}

\subsection{Modules over a semiring}

We now turn our attention to modules over a semiring. % idempotent semimodules. 

\begin{definition}\label{def:sm}
Let $\cis$ be a semiring. 
A \textit{right $\cis$-module} %(or a \textit{right $\cis$-semimodule}) 
is a commutative monoid $(M, \oplus, 0)$ equipped with a right action $M \times \cis \ni (x, t) \mapsto x.t \in M$ such that, for all $x \in M$, $x.0 = 0$, $x.1 = x$, and for all $y \in M$, $s, t \in \cis$, 
\begin{align*}
x.(s t) &= (x.s).t, \\ 
(x \oplus y).t &= x.t \oplus y.t, \\ 
x.(s \oplus t) &= x.s \oplus x.t. %\\ 
%x.0 &= 0, \\ 
%x.1 &= x. 
\end{align*}
A subset of $M$ is a \textit{submodule} if it contains $0$ and is closed under addition and external multiplication.  
\end{definition}
In the sequel we shall say \textit{$\cis$-module} or \textit{module over $\cis$} for \textit{right $\cis$-module}, and we shall only deal with modules over an idempotent semifield $\cis$. 
Then the previous axioms imply that $x \oplus x = x$ for all $x \in M$, and $0.t = 0$ for all $t \in \cis$, so the addition $x \oplus y$ of two elements $x, y$ of $M$ is the supremum of $\{x, y\}$ with respect to the induced partial order $x \leqslant y \Leftrightarrow x \oplus y = y$. In other words, $(M, \oplus, 0)$ is a semilattice. % represents a supremum. % so we may write $x \oplus y$ instead of $x + y$. 
%Similarly to dioids, 

%\textcolor{myblue}{le point suivant est important pour bien définir les modules complets... }

For background on -or applications of- modules over dioids or quantales, see %\cite{Joyal84}, 
Zimmermann \cite{Zimmermann77}, Samborski{\u\i} and Shpiz \cite{Samborskii92}, Abramsky and Vickers \cite{Abramsky93}, Rosenthal \cite{Rosenthal94}, \cite{Huth97}, Kruml \cite{Kruml02}, Cohen et al.\ \cite{Cohen04}, Litvinov et al.\ \cite{Litvinov01}, Shpiz \cite{Shpiz07}, Shpiz and Litvinov \cite{Shpiz07}, Gondran and Minoux \cite{Gondran08}, Russo \cite{Russo10}, Castella \cite{Castella08}. 

\begin{remark}
Some authors, especially in the area of idempotent analysis and max-plus algebra, prefer to call \textit{semimodule} a module over a semiring, and \textit{idempotent semimodule} a module over a dioid or over an idempotent semifield. %(over a semiring) rather than the term \textit{module}. 
However, one can see that the axioms given in Definition~\ref{def:sm} do not differ from the axioms defining a classical module (over a ring), and distinctions only appear in the choice of the base semiring. 
The same remark can be made for axioms defining a \textit{morphism} between modules (see Section~\ref{sec:mf} for the precise definition). Hence, from a categorical (and also from a historical) point of view, we see no reason not to keep on with the term \textit{module}. 
\end{remark}

%\textcolor{myblue}{cf. aussi LESCOT09 qui parle de $\mathbb{B}$-module}

%\begin{example}
%Every dioid (resp.\ complete dioid) %complete dioid) 
%$\cis$ can be seen as a $\cis$-module (resp.\ a complete $\cis$-module). %, a complete $\cis$-module). 
%But sometimes it can be more relevant to consider $\cis$ as a $\{ 0,1 \}$-module, where $\{ 0,1 \}$ is the natural two-elements dioid. %; in that case $\cis$ is complete as a dioid if and only if it is complete as a $\{0, 1\}$-module. %More generally, every dioid $\cis$ can considered as a module over an arbitrary dioid $D$, whenever $D$ has no zero divisors, if one defines, for all $x \in \cis$, $x.t = x$ if $t \in D \setminus \{0\}$ and $x.0 = 0$. 
%\end{example}

The following example is inspired by extreme value theory. % (see e.g.\ Resnick \cite{Resnick87}). 

\begin{example}\label{ex:evt}
We equip the set $\mathbb{R}_+$ of nonnegative real numbers with its idempotent semifield structure, i.e.\ with the maximum operation for $\oplus$, and the usual multiplication. We write $\scit = (\scif, \oplus, \times)$. 
Let $\mu, \sigma, \xi$ be real numbers with $\sigma > 0$, and consider 
$$
M_{\mu, \sigma, \xi} = \left\{ x \in \mathbb{R} \cup \{ -\infty \}	 : 1 + \xi \frac{x - \mu}{\sigma} > 0 \right\}. 
$$
Then $(M_{\mu, \sigma, \xi}, \oplus, \mathbf{0})$ is an $\scit$-module if $\oplus$ denotes the usual maximum operation, if $\mathbf{0}$ denotes $\mu - \sigma/\xi$ if $\xi$ is positive, $-\infty$ otherwise, and if we consider the external multiplication defined by 
$$
x.t = \mu - \frac{\sigma}{\xi} + \frac{\sigma t^{\xi}}{\xi} (1 + \xi \frac{x - \mu}{\sigma}), 
$$
if $\xi$ is non-zero, and 
$$
x.t = x + \sigma \log(t) 
$$
otherwise, for all $x \in M_{\mu, \sigma, \xi}$, $t \in \mathbb{R}_+$. %Moreover, the least element of 
\end{example}

\section{Morphisms and linear forms}\label{sec:mf}

%In this section we consider a union-complete filter selection $\F$ and a right $\cis$-module $M$ over an idempotent semifield $\cis$. 

%We rely on the concept of maxitive measures to define \textit{linear forms}. Such a ``point-free'' approach may be compared with the works of de Cooman et al.\ \cite{deCooman01} and Comman \cite{Comman03}. These linear forms take values in a dioid, which may not be complete. 
In this section, $\F$ is a union-complete filter selection, and $\cis$ is an idempotent semifield. When continuity assumptions on $\cis$ are required, we use the tools of $\Z$-theory introduced in Section~\ref{secpre2}. 

%. An \textit{order extension} $Q/P$ is a pair $(P,Q)$ such that $Q$ is a complete lattice, $P \subset Q$ is equipped with the induced order, and if $A \subset P$ has a supremum (resp.\ an infimum) in $P$, then it coincides with its supremum (resp.\ its infimum) in $Q$. %The extension is \textit{complete} (resp. \textit{locally complete}) if $P$ is a complete lattice (resp. a locally complete lattice). 

%\begin{definition}\label{defmax1}
A \textit{morphism} (or \textit{linear map}) between two $\cis$-modules $M$ and $N$ is a map $f : M \rightarrow N$ satisfying both following conditions: 
\begin{itemize}
	\item \textit{homogeneity}: $f(x.t) = f(x).t$, 
	\item \textit{maxitivity}: $f(x \oplus y) = f(x) \oplus f(y)$, 
\end{itemize}
for all $x, y \in M$, $t \in \cis$. 
%\end{definition}
Or equivalently, $f(0) = 0$ and 
$
f(x \oplus y.t) = f(x) \oplus f(y).t, 
$
for all $x, y \in M$ and $t \in \cis$.  
A morphism $f$ is \textit{smooth} if, for all $\F$-sets $F$ of $M$ with infimum, $f(F)$ has an infimum in $N$ such that $f(\bigwedge F) = \bigwedge f(F)$. 
A (\textit{smooth}) \textit{linear form} %or (smooth) \textit{linear form} 
on a $\cis$-module $M$ is a (smooth) morphism $v : M \rightarrow \cis$, where $\cis$ is considered as a $\cis$-module. 

\begin{translation}[Smoothness]
\textcolor{white}{}
\begin{enumerate}
	\item \textcolor{myred}{If $\F = \UpperStar$, then a morphism is smooth if and only if it preserves all nonempty existing infima. }
	\item \textcolor{myblue}{If $\F = \Filters$, then a morphism is smooth if and only if it is Scott-continuous. }
  \item \textcolor{mygreen}{If $\F = \PFilters$, then every morphism is smooth. }
\end{enumerate}
\end{translation}

\begin{example}[Example~\ref{ex:evt} continued]\label{ex:evt2}
The map $v : M_{\mu, \sigma, \xi} \rightarrow \mathbb{R}_+$ defined by $v(x) = (1 + \xi \frac{x - \mu}{\sigma})^{1/\xi}$ if $\xi$ is non-zero, $v(x) = \exp(\frac{x -\mu}{\sigma})$ otherwise, is a linear form on $M_{\mu, \sigma, \xi}$, smooth with respect to $\UpperStar$. 
\end{example}

\begin{example}
The set $\mathbb{R}_+$ is still equipped with its idempotent semifield structure.  
A \textit{maxitive measure} on a $\sigma$-algebra $\mathrsfs{B}$ is a map $\nu : \mathrsfs{B} \rightarrow \sci$ such that $\nu(\emptyset) = 0$ and 
$$
%\nu(\bigcup_{j\in J} G_j) = \bigoplus_{j \in J} \nu(G_j), 
\nu(B_1 \cup B_2) = \nu(B_1) \oplus \nu(B_2), %\oplus \nu(G_1 \cap G_2), 
$$
for all $B_1, B_2 \in \mathrsfs{B}$. % every finite family $\{G_j\}_{j\in J}$ of elements of $\call{E}$. 
It is \textit{$\sigma$-maxitive} if it commutes with unions of nondecreasing sequences of elements of $\mathrsfs{B}$. 
The \textit{Shilkret integral} (or \textit{idempotent integral}) of some measurable map $f : E \rightarrow \sci$ with respect to a ($\sigma$-)maxitive measure $\nu$ on $\mathrsfs{B}$ is defined by 
\begin{equation*}%\label{defint}
\dint{E} f . d\nu = \bigoplus_{t \in \mathbb{R}_+} t . \nu(f > t).  
\end{equation*}
Such a map $f$ is \textit{$\nu$-integrable} if its Shilkret integral is finite. Then the set $M$ of $\nu$-integrable maps is an $\scit$-module, and the Shilkret integral is a linear form on $M$. 
\end{example}

%on regarde les fonctions intégrables pour éviter de parler de v à valeurs dans \overline{\mathbb{R}}_+... ce qui n'est pas forcément la meilleure solution de contournement

%\begin{translation}[Smooth ideals]
%\begin{enumerate}
%	\item \textcolor{myred}{If $\F = \UpperStar$, then an ideal is archimedean if and only if it preserves all nonempty existing infima. }
%	\item \textcolor{myblue}{If $\F = \Filters$, then a morphism is smooth if and only if it is Scott-continuous. }
%  \item \textcolor{mygreen}{If $\F = \PFilters$, every morphism is smooth. }
%\end{enumerate}
%\end{translation}

\begin{notations}\label{not:id}
Let $I$ be a subset of a $\cis$-module $M$. %$\overline{M}$. 
\begin{itemize}
	\item For all $t \in \cis\setminus\{0\}$, we write $I.t = \{ x.t : x \in I \}$, and $I.0 = \bigcap_{t \neq 0} I.t$. The reader is warned that $I.0$ does not coincide with $\{ 0 \}$ in general. 
	\item For all $x \in M$, we denote by $\langle I, x \rangle$ the set $\{ t \in \cis : x \in I.t\}$. 
\end{itemize}
\end{notations}

%An \textit{ideal} of a dioid $\cis$ is a lower subset $I$ of $\cis$ such that, for all $s, t \in I$, $s \oplus t \in I$,  $\bigoplus \Phi \in I$, for every nonempty finite subset $\Phi$ of $I$ whose supremum exists in $P$. Such an ideal is not necessarily directed, so this differs from the standard definition (see for instance Gierz et al.\ \cite{Gierz80}). 
%Every ideal $I$ of $M$ is \textit{convex}, in the sense that, for all $s, t \in \cis$, $I.(s \oplus t) = I.s \oplus I.t$. 
%Note that every lower convex subset is an ideal. We call such a subset a \textit{convex ideal} of $M$. For instance, every principal ideal $\downarrow\!\! c$, $c \in M$, is a convex ideal of $M$. 
 % \textit{ideal} of a dioid $\cis$ is a lower subset $I$ of $\cis$ such that, for all $s, t \in I$, $s \oplus t \in I$,  $\bigoplus \Phi \in I$, for every nonempty finite subset $\Phi$ of $I$ whose supremum exists in $P$. 

%If $v$ is a linear form on $M$, then $I(v) = \{ x \in M : 1 \geqslant v(x) \}$ is an ideal of $M$. 

%Let $M$ be a module over $\cis$. 
A subset $X$ is \textit{lower} if $X = \downarrow\!\!X$, where $\downarrow\!\!X := \{ y : \exists x \in X, y \leqslant x \}$. An \textit{ideal} $I$ of $M$ is a lower subset of $M$ such that $x \oplus y \in I$, for all $x, y \in I$. 
An ideal $I$ is \textit{smooth} if, for all $\F$-sets $F$ of $M$ with infimum, $\bigwedge F \in I$ implies $F \cap I \neq \emptyset$. %It is \textit{archimedean} if 
%\begin{itemize}
%	\item for all $x \in M$, there is some $t \in \cis$ such that $x \in I.t$; 
%	\item for all $t \in \cis$ such that $1 \gg t$, $I.t \subset \twoheaddownarrow I$. 
%\end{itemize}
An ideal $I$ is \textit{right-continuous} if $I.t = \bigcap_{s \gg t} I.s$ for all $t \in \cis$, and \textit{left-continuous} if $I.t = \bigcup_{t \gg s} I.s$ for all $t \in \cis$. 

The next proposition, which is inspired by the concept of \textit{Minkowski functional} (or \textit{gauge}) in convex analysis and by a remark of Nguyen et al.\ \cite{Nguyen97} on maxitive measures, provides a generic way of constructing a linear form from an ideal. We first prove a useful lemma. 

\begin{lemma}[Compare with \protect{\cite[Lemma~5.1]{Litvinov01}}]\label{lem:inf0}
Let $\cis$ be an idempotent semifield. Then $\cis \setminus \{0\}$ has an infimum, and 
\begin{itemize}
	\item $\bigwedge \cis \setminus \{0\} = 1$ if and only if $\cis = \{0,1\}$, 
	\item $\bigwedge \cis \setminus \{0\} = 0$ if and only if $\cis \neq \{0,1\}$. 
\end{itemize}
\end{lemma}

\begin{proof}
If $\cis = \{ 0, 1\}$, the result is clear, so suppose that $\cis \neq \{ 0, 1\}$. This implies the existence of some $t \in \cis, t > 0$ and $t \neq 1$. To show that $0$ is the infimum of $\cis \setminus \{0\}$, we pick some lower bound $s$ of $\cis \setminus\{0\}$, and we assume that $s > 0$. If $s = 1$, the definition of $s$ gives $t > 1$, hence $0 < t^{-1} < 1 = s$, a contradiction. As a consequence, $s \neq 1$. Since $1 \in \cis \setminus\{0\}$, we have $s < 1$, by definition of $s$. Hence, $0 < s^2 < s$, another contradiction. We conclude that $s = 0$, which proves that $0$ is the infimum of $\cis \setminus\{0\}$ whenever $\cis \neq \{0, 1\}$. 
\end{proof}

\begin{proposition}\label{nguyen bouchon-meunier3}
%Assume that $\cis$ is filtered-complete. 
%Let a map $v : E \rightarrow \cis$. 
%Let $\cis$ be an idempotent semifield, and 
Let $\cis$ be an idempotent semifield, $\cis \neq \{0, 1\}$, and let $I$ be an ideal (resp.\ a smooth ideal) of $M$ %\overline{M}$ 
such that, for all $x \in M$, $\langle I, x \rangle$ %$\{t \in \cis : x \in I.t \}$ 
is an $\F$-set with infimum. Define $v : M \rightarrow \cis$ by 
\begin{equation}\label{ideals000}
%v(x) = \bigwedge \left\{t \in \cis : x \in I.t \right\}, % \qquad\qquad(\forall G\in\mathrsfs{E}).
v(x) = \bigwedge \langle I, x \rangle, % \qquad\qquad(\forall G\in\mathrsfs{E}).
\end{equation}
for all $x \in M$. 
%If $I$ is right-continuous, in the sense that $I.t = \bigcap_{s \gg t} I.s$ for all $t \in \cis$, then $v$ is a linear form on $M$. 
If $I$ is right-continuous, then $I = \{ 1 \geqslant v \}$ and $v$ is a linear form (resp.\ a smooth linear form) on $M$. 
\end{proposition}

%JE PENSE QU'ON PEUT simplement prendre $I$ un (convex) ideal de $\overline{M}$... 

\begin{proof}
If $v$ is given by Equation~(\ref{ideals000}) with a right-continuous ideal $I$, then $v$ is order-preserving, for if $x \leqslant y$, $y \in I.t$, and $t \neq 0$, then $x.t^{-1} \leqslant y.t^{-1} \in I$, so that $x = (x.t^{-1}).t \in I.t$. Thus, one has $\{t \in \cis : x \in I.t \} \supset \{t \in \cis : y \in I.t \}$, so that $v(x) \leqslant v(y)$. 

Now let us show that $v(x) \oplus v(x') \geqslant v(x \oplus x')$. So let $s \gg v(x) \oplus v(x')$. Then $s \gg v(x)$, so there exists some $t \in \cis$ such that $s \geqslant t$ and $x \in I.t$. There is also some $t'$ with the corresponding properties with respect to $x'$. Note that $x, x' \in I.s$, which implies $x \oplus x' \in I.s$. 
Since $I$ is right-continuous, we have $x \oplus x' \in I.s_0$, where $s_0 := v(x) \oplus v(x')$. If $s_0  = 0$, then $x \oplus x' \in I.0$, so that $v(x \oplus x') = 0 = s_0 = v(x) \oplus v(x')$ by definition of $v$. Otherwise, we can write $x \oplus x' = y.s_0$, with $y \in I$. We thus have $v(x \oplus x') = v(y).s_0$. Since $y \in I$, $v(y) \leqslant 1$. This leads to $v(x \oplus x') \leqslant s_0$, i.e.\ $v(x\oplus x') \leqslant v(x) \oplus v(x')$. 

For $v$ to be a linear form, it remains to show that $v(x.t) = v(x).t$, for all $x \in M$, $t\in \cis$. This step is not difficult and left to the reader. 

It is clear that $I \subset \{ 1 \geqslant v \}$. For the reverse inclusion, let $x \in \{ 1 \geqslant v\}$, i.e.\ $1 \geqslant v(x)$. To prove that $x \in I$, we use the right-continuity of $I$, i.e.\ we show that $x \in I.s$ for all $s \gg 1$. We have $s \gg v(x) = \bigwedge \langle I, x \rangle$. The subset $\langle I, x \rangle$ is assumed to be an $\F$-set, so $s \in \langle I, x \rangle$, i.e.\ $x \in I.s$, which is the desired result. 

Suppose in addition that $I$ is smooth, and let us show that $v$ is smooth. First recall that, if $F$ is an $\F$-set of $M$ with infimum $f_0$ and $t\in \cis\setminus\{0\}$, then $F.t^{-1}$ is an $\F$-set such that $\bigwedge (F.t^{-1}) = f_0.t^{-1}$. We obtain 
\begin{align*}
v(f_0) &= \bigwedge \{ t \in \cis\setminus\{ 0 \} : f_0 \in I.t \} \\
 &= \bigwedge \{ t \in \cis\setminus\{ 0 \} : \bigwedge (F.t^{-1}) \in I \} \\
 &= \bigwedge \bigcup_{f \in F}\{ t \in \cis\setminus\{ 0 \} : f.t^{-1} \in I \}, 
\end{align*}
since $I$ is smooth. We deduce that $v(f_0) = \bigwedge_{f \in F} \bigwedge\{ t \in \cis\setminus\{ 0 \} : f \in I.t \} = \bigwedge_{f \in F} v(f)$, so $v$ is smooth. 
\end{proof}

When the range $\cis$ of the map $v$ is continuous, one can remove the assumption of right-continuity of $I$. This leads to the converse statement as follows.

%\begin{remark}
%A slightly more general result can be given. If $\cis$ is a dioid, the same proof shows that the infimum of $\cis^*$, the subset of invertible elements of $\cis$, equals $0$, if $\cis^*$ is not reduced to $\{1\}$. 
%\end{remark}

\begin{proposition}\label{prop:nbm2}
Assume that $\cis$ is a (stably-)continuous cis. %idempotent semifield, $\cis \neq \{0, 1\}$. 
%Let $\overline{\cis}/\cis$ be a local $\oplus$-domain extension. % which is a bounded-join-semilattice. 
A map $v : M \rightarrow \cis$ is a (smooth) linear form on $M$ if and only if there is some (smooth) ideal $I$ of $M$ such that $\langle I, x \rangle$ %$\{t \in \cis : x \in I.t \}$ 
is an $\F$-set and %with infimum and 
\begin{equation*}%\label{ideals1}
v(x) = \bigwedge \langle I, x \rangle, 
\end{equation*}
for all $x \in M$. 
In this case: 
\begin{enumerate}
	\item\label{nbm21} $I$ is right-continuous if and only if $I = \{ x \in M : 1 \geqslant v(x) \}$; % for all $t \in \cis$, 
	\item\label{nbm22} %if the way-above relation is multiplicative, then 
	$I$ is left-continuous if and only if $I = \{ x \in M : 1 \gg v(x) \}$. % for all $t \in \cis$. 
\end{enumerate}
\end{proposition}

\begin{proof}
At first we consider the case where $\cis$ is a continuous cis. 
If $v$ is a linear form, define $I := \{ 1 \geqslant v \}$. This is an ideal such that, for all $t \neq 0$, $I.t = \{ t \geqslant v \}$, and, by Lemma~\ref{lem:inf0}, $I.0 = \{ x \in M : v(x) = 0 \}$. Since $\cis$ is continuous, $I$ is right-continuous. Moreover, $\langle I, x \rangle$ equals the principal filter generated by $v(x)$, hence is an $\F$-set, and $v(x) = \bigwedge \langle I, x \rangle$ for all $x$. 

If $\cis$ is stably-continuous and $v$ is smooth, we can rather define $I := \{ 1 \gg v \}$. By hypothesis the way-above relation is additive, so that $I$ is an ideal. Also, $I.t = \{ t \gg v \}$ by Lemma~\ref{lem:akiansinger}, and $I.0 = \{ x \in M : v(x) = 0 \}$ since $\cis$ is continuous. Left-continuity of $I$ holds by the interpolation property. Moreover, for all $x$, $\langle I, x \rangle = \twoheaduparrow v(x)$, which is an $\F$-set whose infimum is $v(x)$ since $\cis$ is continuous. Smoothness of $I$ is a consequence of the smoothness of $v$ and of the fact that $\langle I, x \rangle$ is nonempty. 

%Conversely, assume that Equation~(\ref{ideals0}) is satisfied, and let us show that $v$ is a linear form. \ciset $J = \bigcap_{s \gg 1} I.s$. Then $J$ is an ideal of $M$ containing $I$. We shall now use the interpolation property to show that $J$ is right-continuous. So let $x \in \bigcap_{t \gg u} J.t$. We assume that $u \neq 0$, and we want to show that $x \in J.u$, i.e.\ that $x \in I.su$ for all $s \gg 1$. If $s \gg 1$, then $su \gg u$ (see e.g.\ Akian and Singer \cite[Lemma~2.1]{Akian03}), so we can pick some $t \in \cis$ such that $su \gg t \gg u$. Then $s u = s' t$, where $s' := s u t^{-1} \gg 1$ (see again \cite[Lemma~2.1]{Akian03}). Since $x \in \bigcap_{t \gg u} J.t$, we deduce that $x \in I.s't =  I.su$. This implies that $x \in J.u$. The case where $u=0$ can be treated similarly. 
%By continuity of $\cis$ one has $v(x) =  \bigwedge\{t \in \cis : x \in J.t\}$ \textcolor{myblue}{POURQUOI ?}. Using Proposition~\ref{nguyen bouchon-meunier3}, $v$ is a linear form. \textcolor{myblue}{INSUFFISANT ! IL faut calculer $J \setminus x$, montrer que c'est un $F$-set d'inf égal à $x$ !}

%\textcolor{myred}{JE RECRIS LE § CI-DESSUS: 
Conversely, assume that Equation~(\ref{ideals000}) is satisfied, and let us show that $v$ is a linear form. Let $J = \bigcap_{s \gg 1} I.s$. Then $J$ is an ideal of $M$ containing $I$. We prove that, for all $t \in \cis$, 
\begin{equation}\label{eq:eq}
t \geqslant v(x) \Leftrightarrow x \in J.t. 
\end{equation}
Using Lemma~\ref{lem:inf0}, it suffices to prove Equivalence~(\ref{eq:eq}) for  $t \neq 0$. If $t \geqslant v(x)$ and $s \gg 1$, then $s t \gg t$ by Lemma~\ref{lem:akiansinger}, %(see e.g.\ Akian and Singer \cite[Lemma~2.1]{Akian03}), 
so $s t \gg v(x)$. This gives $x \in I.s t$, hence $x.t^{-1} \in I.s$, for all $s \gg 1$. By definition of $J$ we get $x.t^{-1} \in J$, i.e.\ $x \in J.t$. Now we suppose that $x \in J.t$, and we want to show that $t \geqslant v(x)$. So let $s \gg t$. If $u = s t^{-1}$, then $u \gg 1$ (see again Lemma~\ref{lem:akiansinger}), % \cite[Lemma~2.1]{Akian03}), 
so $x t^{-1} \in I.u$. Thus, $x \in I.s$. The definition of $v$ implies $s \geqslant v(x)$, for all $s \gg t$. By continuity of $\cis$ we have $t \geqslant v(x)$. So Equivalence~(\ref{eq:eq}) is proved. This also shows that $J$ is right-continuous and that $\langle J, x \rangle$ is an $\F$-set whose infimum is $v(x)$, for all $x \in M$. By Proposition~\ref{nguyen bouchon-meunier3}, $v$ is a linear form, and, as in the proof of Proposition~\ref{nguyen bouchon-meunier3}, $v$ is smooth if $I$ is smooth. 

%\textcolor{myred}{PROBLEME : $I$ smooth n'implique pas $J$ smooth...}

To finish the proof, suppose again that $v$ is a linear form defined by Equation~(\ref{ideals000}). 
If $I$ is right-continuous, then the previous point implies $I = J$ and $I.t = \{ t \geqslant v \}$, for all $t \in \cis$, so Item~(\ref{nbm21}) is proved. %, let us show Item~(\ref{nbm22}). 
If $I$ is left-continuous, the inclusion $I \supset \{ 1 \gg v \}$ is clear, by definition of $v$ and $\gg$. If $x \in I$, then $x \in I.s$ for some $s \in \cis$ such that $1 \gg s$, by left-continuity of $I$. This implies that $s \geqslant v(x)$, so that $1 \gg v(x)$, and Item~(\ref{nbm22}) is proved. 
\end{proof}

\begin{translation}
Back to the three main instances of filter selections, the assumptions of Proposition~\ref{prop:nbm2} translate as follows. 
%we have the following sufficient conditions for applying Proposition~\ref{nguyen bouchon-meunier2}. 
\begin{enumerate}
	\item \textcolor{myred}{%Assume that $\F = \UpperStar$. 
	%In $\cis$, %is a completely distributive cis, 
	$\langle I, x \rangle$ is an $\UpperStar$-set if and only if $\langle I, x \rangle$ is nonempty. This condition is satisfied for all $x \in M$ as soon as $I \neq \{0\}$ and, for all $x, y \in M$ with $y \neq 0$, there exists some $t \in \cis$ with $x \leqslant y.t$. }
	\item \textcolor{myblue}{%Assume that $\F = \Filters$. 
	%In $\cis$, %is a continuous cis with respect to $\Filters$, then, for all $x \in M$, 
	$\langle I, x \rangle$ is a $\Filters$-set if and only if $\langle I, x \rangle$ is nonempty. As above, this condition is satisfied for all $x \in M$ as soon as $I \neq \{0\}$ and, for all $x, y \in M$ with $y \neq 0$, there exists some $t \in \cis$ with $x \leqslant y.t$. }
	\item \textcolor{mygreen}{%Assume that $\F = \PFilters$. 
	%In $\cis$, %is an idempotent semifield, then it is automatically continuous with respect to $\PFilters$, and, for all $x\in M$, 
	$\langle I, x \rangle$ is a $\PFilters$-set if and only if $\langle I, x \rangle$ has a least element.}% = \uparrow\!\! t_x$ for some $t_x \in \cis$.}
\end{enumerate}
\end{translation}

The last case leads to the following corollary. 

\begin{corollary}\label{cor:equiv}%[Compare with Cohen et al.\ \protect{\cite[??]{Cohen04}}]
%Assume that $\cis$ is a idempotent semifield. 
%Let $\overline{\cis}/\cis$ be a local $\oplus$-domain extension. % which is a bounded-join-semilattice. 
Let $\cis$ be an idempotent semifield, $\cis \neq \{0, 1\}$. 
A map $v : M \rightarrow \cis$ is a linear form on $M$ if and only if there is some ideal $I$ of $M$ such that 
$$
x \in I.t \Longleftrightarrow t \geqslant v(x), 
$$
for all $x \in M$, $t \in \cis$. 
%, for all $x \in M$, $\{t \in \cis : x \in I.t \}$ is a principal filter and 
%\begin{equation*}%\label{ideals1}
%v(x) = \bigwedge \left\{t \in \cis : x \in I.t \right\}.% \qquad\qquad(\forall G\in\mathrsfs{E}).
%\end{equation*}
In this case, $I$ equals $\{ 1 \geqslant v \}$. %, while $I_t = \{ g \in E : t \gg v(g) \}$ for all $t \in \cis$ if and only if $(I_t)$ is left-continuous, i.e. $I_t = \bigcup_{t \gg s} I_s$. 
\end{corollary}

%\textcolor{myblue}{QUESTION: si $M$ est meet-continu et si $\overline{M}$ est la complétion de Dedekind--MacNeille de $M$, est-ce que $\overline{M}/M$ est meet-continu ?}

%For the case where $\cis = \{0, 1\}$, a map $v : M \rightarrow \{ 0, 1 \}$ is a linear form on $M$ considered as a $\{0, 1\}$-module if and only if $v^{-1}(0)$ is a prime ideal of $M$. 

\begin{proof}
Let $\F$ be the filter selection $\PFilters$ that selects principal filters. With this choice, $\cis$ is continuous, and the way-above relation coincides with $\geqslant$, so that $I$ is necessarily right-continuous. To conclude, use Proposition~\ref{prop:nbm2}. 
\end{proof}

%\begin{remark}
%If $\cis$ is an idempotent semifield, $\cis \neq \{0, 1\}$, then the previous corollary brings out a Galois connection between linear forms $v$ on $M$ and ideals $I$ of $M$ such that $\langle I, x \rangle$ is a principal filter in $\cis$ for all $x \in M$, given by 
%$$
%I \subset I(v) \Longleftrightarrow v \leqslant v_I, 
%$$
%where $I(v) := \{ x \in M : 1 \geqslant v(x)\}$ and $v_I(x) := \bigwedge \langle I, x \rangle$. 
%\end{remark}

At this early stage, the reader may already understand, from the previous proof, where the ``paradox'' evoked in the Introduction comes from: no continuity assumption seems to be needed in the terms of Corollary~\ref{cor:equiv}, but this is simply due to the fact that, with respect to the filter selection $\PFilters$, the cis $\cis$ is always continuous. This will be made even clearer in Section~\ref{sec:continuous}, where we shall deal with the representation of continuous linear forms.

\section{Continuous linear forms on a complete module}\label{sec:continuous}

\subsection{Continuity, residuation}

In this section, we prove a representation theorem for continuous linear forms on a complete module. 
Let $\F$ be a union-complete filter selection, and let $M, N$ be modules over an idempotent semifield $\cis$. 
A morphism $f : M \rightarrow N$ is %\textit{smooth} if, for every $\F$-set $F$ of $M$ with infimum, $f(F)$ has an infimum in $N$ such that $f(\bigwedge F) = \bigwedge f(F)$, and 
\textit{continuous} if it is smooth and such that, for every subset $X \subset M$ with a supremum in $M$, $f(X)$ has a supremum in $N$ satisfying $f(\bigoplus X) = \bigoplus f(X)$. 

\begin{translation}[Continuity]
\textcolor{white}{}
\begin{enumerate}
	\item \textcolor{myred}{If $\F = \UpperStar$, then a morphism is continuous if and only if it preserves all existing infima and suprema.}
	\item \textcolor{myblue}{If $\F = \Filters$, then a morphism is continuous if and only if it is bi-Scott-continuous (this notion was called \textit{wo-continuity} by Shpiz \cite{Shpiz07}).}
  \item \textcolor{mygreen}{If $\F = \PFilters$, then a morphism is continuous if and only if it preserves all existing suprema (this notion of continuity is the one adopted by Cohen et al.\ \cite{Cohen04}; Litvinov et al.\ \cite{Litvinov01} and Shpiz \cite{Shpiz07} called such a morphism a \textit{b-morphism}).}
\end{enumerate}
\end{translation}

%An element $x$ of $M$ is \textit{absorbing} if, for all $t \neq 0$, $x.t = x$. Such an element may differ from both $0$ and $\top$. For instance, in the $\scif$-module of $\sci$-valued lsc maps on a topological space, every lsc map $f$ such that $f(x) \in \{ 0, \infty \}$ for all $x$ is absorbing. 
We say that $M$ is \textit{completable} %(resp.\ \textit{complete}) 
if, for all $x \in M$, the map $\cis \rightarrow M, t \mapsto x.t$ is a continuous morphism, i.e.\ if, for all $x \in M$ and all $T \subset \cis$ with supremum, 
\begin{equation}\label{eq:supmod}
x . \bigoplus T = \bigoplus_{t \in T} x.t, %\qquad (\bigoplus X).t = \bigoplus_{x \in X} x.t,
\end{equation}
and if for all $x \in M$ and all $\F$-sets $F$ of $\cis$ with infimum, 
\begin{equation}\label{eq:supmodf}
x . \bigwedge F = \bigwedge_{f \in F} x.f. %, \qquad (\bigwedge F).t = \bigwedge_{f \in F} f.t, 
\end{equation}

\begin{remark}\label{rk:infreached}
Note that, if Equation~(\ref{eq:supmod}) is satisfied for every $T \subset \cis$ with supremum, then Equation~(\ref{eq:supmodf}) is also satisfied for every $F \subset \cis$ with \textit{non-zero infimum} (one does not need $F$ be to an $\F$-set in this case). Hence in the definition of completability the role of Equation~(\ref{eq:supmodf}) is to control the behaviour of $\cis \rightarrow M, t \mapsto x.t$ around zero. Therefore, \textcolor{myred}{the case $\F = \UpperStar$ is demanding}, while \textcolor{mygreen}{with $\F = \PFilters$ this behaviour is unconstrained}. 
\end{remark}

%\textcolor{blue}{EXPLIQUER QUE c'est une affaire de régularité en 0, puisque l'égalité ci-dessus est vérifiée pour TOUT F d'inf non nul. <}
Also, $M$ is \textit{complete} if it is completable and such that every upper-bounded subset (resp.\ every $\F$-set) %(resp.\ arbitrary) 
has a supremum (resp.\ an infimum). %every $\F$-subset of $M$ has an infimum, 
%The latter equality always holds if $\cis$ is an idempotent semifield. 
%Modules over a dioid are examined in , ... ?  
In Section~\ref{sec:dmcc}, Theorem~\ref{thm:nc} will define the concept of \textit{normal completion} of a completable module and show that completability is equivalent to embeddability into a complete module. 

A map $f : M \rightarrow N$ between $\cis$-modules $M, N$ is \textit{residuated} if there exists a (necessarily unique) map $f^{\#} : N \rightarrow M$, called the \textit{adjoint} of $f$, and 
satisfying 
\begin{equation*}%\label{eq:res}
x \leqslant f^{\#}(y) \Longleftrightarrow y \geqslant f(x), 
\end{equation*}
for all $x\in M$, $y \in N$. 
Residuated maps are related to Galois connections, see Ern\'e et al.\ \cite{Erne93}. 
A \textit{residuated form} on $M$ is a homogeneous residuated map from $M$ to $\cis$. 
A map $v : M \rightarrow \cis$ is \textit{non-degenerate} if $\{ x \in M  : 1 \geqslant v(x) \}$ is upper-bounded. 
For instance the map $M \ni x \mapsto 0 \in \cis$ is a non-degenerate (continuous) linear form if and only if $M$ has a greatest element. 

\begin{lemma}\label{lem:rescont}
Let $M$ be a complete module over a cis $\cis$. Then a map $v : M \rightarrow \cis$ is a smooth residuated form if and only if it is a non-degenerate continuous linear form. 
\end{lemma}

\begin{proof}
Necessity is clear. For sufficiency, let $v$ be a non-degenerate continuous linear form, and let $w : \cis \rightarrow M, t \mapsto \bigoplus \{ x \in M : t \geqslant v(x) \}$. This map is well-defined since $v$ is non-degenerate and homogeneous. 
If $t \geqslant v(x)$, then $x \leqslant w(t)$ by definition of $w$. Conversely, if $x \leqslant w(t)$, then $v(x) \leqslant v(w(t))$,  and since $v$ preserves arbitrary existing suprema, $v(x) \leqslant \bigoplus \{ v(x') : x' \in M, t \geqslant v(x') \} \leqslant t$. This proves that $w$ is the adjoint of $v$, hence $v$ is a (smooth) residuated form. 
\end{proof}

\begin{remark}
\textcolor{mygreen}{If $\F = \PFilters$, the previous lemma identifies residuated forms with non-degenerate \textit{b-linear functionals} in the sense of Litvinov et al.\ \cite{Litvinov01}.}
\end{remark}

%Assume that $\cis$ is a cis and consider the complete dioid $\overline{\cis} = \cis \cup \{\top\}$ extending $\cis$ (see the definitions in Example~\ref{ex:barl}). 
%, in the sense that every nonempty subset of $\cis$ has an infimum. If $\cis \neq \{0,1\}$, then $\cis$ has no top element, in which case we add some element $\top$ to $\cis$ and define $\overline{\cis} = \cis \cup \{\top\}$. Naturally extending $\oplus$ and $\times$ to $\overline{\cis}$ by $x \oplus \top = \top \oplus x = \top$ and $x.\top = \top .x = \top$ if $x \neq 0$, $0.\top = \top.0 = 0$ otherwise, we see that $\overline{\cis}$ has the structure of a complete dioid. 
%We also define $\top^{-1} = 0$ and $0^{-1} = \top$. 
%We also assume that $\cis$ is \textit{up-complete}, in the sense that every nonempty subset of $\cis$ has a supremum. 
%Let $\overline{M}/M$ is a extension over a cis $\cis$. 

\subsection{Archimedean elements and scalar product}\label{ssec:archi}

As explained in the paragraph after Lemma~\ref{lem:akiansinger}, a cis $\cis$ has no greatest element. However, we can conventionally add a top $\top$ to $\cis$ and define $\overline{\cis} = \cis \cup \{ \top \}$. Naturally extending $\oplus$ and $\times$ to $\overline{\cis}$ by $t \oplus \top = \top \oplus t = \top$, $t.\top = \top .t = \top$ if $t \neq 0$, and $0.\top = \top.0 = 0$, we see that $\overline{\cis}$ has the structure of a complete dioid. 
We also define $\top^{-1} = 0$ and $0^{-1} = \top$. 
%Assume that $\cis$ is a cis and consider the complete dioid $\overline{\cis} = \cis \cup \{\top\}$ extending $\cis$ (see the definitions in Example~\ref{ex:barl}). 
%, in the sense that every nonempty subset of $\cis$ has an infimum. If $\cis \neq \{0,1\}$, then $\cis$ has no top element, in which case we add some element $\top$ to $\cis$ and define $\overline{\cis} = \cis \cup \{\top\}$. Naturally extending $\oplus$ and $\times$ to $\overline{\cis}$ by $x \oplus \top = \top \oplus x = \top$ and $x.\top = \top .x = \top$ if $x \neq 0$, $0.\top = \top.0 = 0$ otherwise, we see that $\overline{\cis}$ has the structure of a complete dioid. 
If $x, c \in M$, we let 
$$
x\backslash c = \bigoplus \{ t \in \cis : c \geqslant x.t \}
$$ 
whenever this set is upper-bounded, and $x \backslash c = \top$ otherwise. 
Also, the infimum of the subset $F_c(x) = \{ t \in \cis : c.t \geqslant x \}$ in $\overline{\cis}$ is denoted by $\langle c, x \rangle$, and one can easily check that   
$$
\langle c, x \rangle = (x\backslash c)^{-1}, 
$$
for all $x, c \in M$. 
We are interested in conditions on $c$ ensuring that the map $x \mapsto \langle c, x \rangle$ is a continuous or a residuated linear form on $M$. 
So we need the  

\begin{definition}\label{def:archi}
Let $M$ be a module over an idempotent semifield $\cis$. An element $c \in M$ is called \textit{archimedean} if both following conditions are satisfied: 
\begin{itemize}
	\item the subset $F_c(x)$ is nonempty for all $x \in M$, 
	\item if the infimum of $F_c(x)$ is zero, then $F_c(x) \supset \cis \setminus \{0\}$, 
\end{itemize} 
where $F_c(x)$ denotes the subset $\{ t \in \cis : c.t \geqslant x \}$. 
\end{definition}

The first condition implies that the bracket $\langle c, x \rangle$ is well-defined in $\cis$ for all $x \in M$. 
The second condition may seem unnatural and so deserves some explanation. First observe that it will automatically be satisfied in any of the following standard situations: 
\begin{itemize}
	\item if $\F \in \{\UpperStar, \Filters\}$ and $M$ is completable, 
	\item if the cis $\cis$ is continuous with respect $\Filters$, 
	\item in particular if $\cis$ is a totally ordered cis (e.g.\ $\cis = \scit$). 
\end{itemize}

To see how it works, let us suppose in the following lines that $M$ is a completable module over a cis $\cis$, and that $c$ is such that $F_c(x)$ is nonempty for all $x \in M$. 
%This implies that the bracket $\langle c, x \rangle$ is well-defined in $\cis$ for all $x \in M$. 
%This implies that $\langle c, x \rangle$ is $\cis$-valued for all $x\in M$. 
%the map $x \mapsto \langle c, x \rangle$ is $\cis$-valued hence is a linear form on $M$. 
%We are interested in conditions on $c$ ensuring that the map $x \mapsto \langle c, x \rangle$ is a continuous or even a residuated linear form on $M$. 
%It would be even nicer if $x \mapsto \langle c, x \rangle$ were residuated. 
How far is then $x \mapsto \langle c, x \rangle$ from being a residuated map?

If $\langle c, x \rangle$ is non-zero, then the infimum of the subset $F_c(x)$ is reached (see Remark~\ref{rk:infreached}), so that $c.\langle c, x \rangle \geqslant x$ or, in other words, 
\begin{equation}\label{eq:residsp}
x \leqslant c.t  \Longleftrightarrow t \geqslant \langle c, x \rangle, 
\end{equation}
for all $t \in \cis$. 
However, this equivalence is no longer guaranteed if $\langle c, x \rangle = 0$. This is where a ``smooth'' behaviour of $\cis \rightarrow M, t \mapsto x.t$ around zero is needed, in accordance with Remark~\ref{rk:infreached}; so at this stage we must distinguish between the different filter selections. 

The important fact is that the subset $F_c(x)$ is always filtered (if non\-empty) by Equation~(\ref{eq:cisinf}). \textcolor{myblue}{This makes $\Filters$ \textit{the most natural filter selection} to use on $\cis$-modules}.
%, so that the filter selection $\Filters$ plays a special role. 
As a consequence, if $\F \in \{\UpperStar, \Filters \}$, Equivalence~(\ref{eq:residsp}) is satisfied by completability of $M$, even if $\langle c, x \rangle = 0$. Thus, the second condition in Definition~\ref{def:archi} is fullfilled, and the map $x \mapsto \langle c, x \rangle$ is residuated. Moreover, $\langle c, x \rangle = 0$ implies $x = 0$. 

\textcolor{mygreen}{The case $\F = \PFilters$ is more delicate}. 
We might include in the definition of $c$ that $F_c(x)$ be an $\F$-set, so here a principal filter; but this would imply that $\langle c, x \rangle \neq 0$ whenever $x \neq 0$, a property that is not desirable for applications (see e.g.\ the case of the Riesz representation theorem in Section~\ref{sec:riesz}). 
That is why we introduced a second \textit{ad hoc} condition in the definition of an archimedean element. 
The following proposition gives sufficient conditions on $\cis$ for this condition to hold. 
%Another possibility is to require an additional condition on $\cis$ in order that every filter with an infimum equal to zero coincides with either $\cis\setminus\{0\}$ or $\cis$. This is the purpose of the next proposition to examine the appropriate hypothesis. 

\begin{proposition}
Let $\cis$ be an idempotent semifield. Consider the following conditions: 
\begin{enumerate}
%	\item\label{cond5} every chain $C$ with a zero infimum satisfies $\uparrow\!\! C \supset \cis\setminus\{0\}$; 
	\item\label{cond1} $1$ is way-above $0$ with respect to $\Filters$; 
	\item\label{cond2} there is a $t \in \cis$ way-above $0$ with respect to $\Filters$; 
%	\item there exists some $t \in \cis$ way-above $0$ with respect to $\Filters$; 	
	\item\label{cond3} every filter with a zero infimum contains $\cis\setminus\{0\}$; 
	\item\label{cond4} every unbounded ideal coincides with $\cis$; 
	\item\label{cond7} $\cis$ is continuous with respect to $\Filters$; 
	\item\label{cond6} $\cis$ is totally ordered; 
\end{enumerate}
Then (\ref{cond6}) $\Rightarrow$ (\ref{cond7}) $\Rightarrow$ (\ref{cond4}) $\Leftrightarrow$ (\ref{cond3}) $\Leftrightarrow$ (\ref{cond2}) $\Leftrightarrow$ (\ref{cond1}). % $\Rightarrow$ (\ref{cond5}). 
If any of these conditions is satisfied, then 
\begin{itemize}
	\item $\cis$ is join-continuous (with respect to $\Filters$), i.e.\ satisfies $s \oplus \bigwedge F = \bigwedge (s \oplus F)$ for all filters $F$ and $s \in \cis$; 
	\item the second condition of Definition~\ref{def:archi} always holds; 
	\item for all $t \in \cis \setminus \{0\}$ and $s < 1$, there is some $n \in \mathbb{N}$ such that $t \geqslant s^n$.  
\end{itemize}
\end{proposition}

An archimedean element $c$ of $M$ is \textit{strongly archimedean} if $t \gg s$ implies $c.t \gg c.s$ for all $s, t \in \cis$. 
For an archimedean element $c$, the map $x \mapsto \langle c, x \rangle$ is smooth only if $c$ is strongly archimedean; the converse statement holds as soon as $\cis$ is continuous.  

The following result justifies the term \textit{scalar product} for the bracket $\langle \cdot, \cdot \rangle$ (see also Cohen et al.\ \cite[Section~3]{Cohen04} for more on this topic; note that these authors preferred to call scalar product the bracket $(\cdot \backslash \cdot)$ rather than $\langle \cdot, \cdot \rangle$). 

\begin{lemma}\label{lem:arch1}
Let $M, N$ be modules over a cis $\cis$, %, let $c \in N$, 
and let $f : M \rightarrow N$ be a (smooth) residuated linear map. % which is both smooth and residuated. 
If $c$ is a (strongly) archimedean element of $N$ then $f^{\#}(c)$ is a (strongly) archimedean element of $M$, and we have 
$$
\langle c, f(x) \rangle = \langle f^{\#}(c), x \rangle, 
%(f(x) \backslash c) = (x \backslash f^{\#}(c)), 
$$
for all $x \in M$, where $f^{\#}$ denotes the upper adjoint of $f$. 
\end{lemma}

\begin{proof}
If $t \in \cis\setminus \{0\}$, we have $c.t \geqslant f(x) \Leftrightarrow c \geqslant f(x.t^{-1}) \Leftrightarrow f^{\#}(c) \geqslant x.t^{-1} \Leftrightarrow f^{\#}(c).t \geqslant x$. Since $c$ is archimedean, $F_c(f(x))$ contains some non-zero element, so that $F_{f^{\#}(c)}(x)$ is nonempty, so the first condition for $f^{\#}(c)$ to be archimedean is checked. 
Now if the subset $F_{f^{\#}(c)}(x)$ has zero infimum, then either it contains $0$ (and in this case $x = 0$ so that $F_{f^{\#}(c)}(x) = \cis)$ or it does not. In the latter case, the series of equivalence at the beginning of the proof shows that $F_{f^{\#}(c)}(x) = F_c(f(x)) \setminus \{0\}$. This implies that $F_c(f(x))$ has zero infimum, thus contains $\cis\setminus\{0\}$ since $c$ is archimedean. Therefore, $F_{f^{\#}(c)}(x)$ also contains $\cis\setminus\{0\}$, and we have proved that $f^{\#}(c)$ is archimedean. 
Using again the equivalence $c.t \geqslant f(x) \Leftrightarrow f^{\#}(c).t \geqslant x$ for all $t \in \cis \setminus\{0\}$, we deduce that $t \geqslant \langle c, f(x) \rangle \Leftrightarrow t \geqslant \langle f^{\#}(c), x \rangle$ for all $t \in \cis \setminus \{0\}$, so 
$$
\langle c, f(x) \rangle = \langle f^{\#}(c), x \rangle. 
$$

For the rest of the proof, assume that $f$ is smooth and that $c$ is strongly archimedean. Let $t \gg s$, and let us show that $f^{\#}(c).t \gg f^{\#}(c).s$. For this purpose, let $F$ be an $\F$-set of $M$ with infimum such that $f^{\#}(c).s \geqslant \bigwedge F$. Then $f^{\#}(c.s) \geqslant \bigwedge F$, hence $c.s \geqslant f(\bigwedge F)$. Since $f$ is smooth, this implies that $c.s \geqslant \bigwedge f(F)$. Since $\uparrow\!\! f(F)$ is an $\F$-set of $\cis$ and $c.t \geqslant c.s$, we obtain $c.t \geqslant f(x)$ for some $x \in F$. This gives $f^{\#}(c).t = f^{\#}(c.t) \geqslant x$, and the result is proved. 
%
%$t \geqslant v(x)$ for some $x \in F$. Since $t \neq 0$, we have $1 \geqslant v(x.t^{-1})$, so by definition of $c$ we have $c \geqslant x.t^{-1}$. Finally, $c.t \geqslant x$, and we have proved that $c.t \gg c$. The claim that $c$ is strongly archimedean follows. 
\end{proof}

We are now in a position to prove the main result of this section. %shall show that, conversely, every residuated form admits the previous representation, as soon as $M$ is complete. 

\begin{theorem}[Compare \protect{\cite[Theorems~5.1-5.2]{Litvinov01}}, \protect{\cite[Corollary~39]{Cohen04}}]\label{thm:litv} %[\protect{\cite[Theorem~5.2]{Litvinov01}}]\label{thm:litv}
Suppose that $M$ is a complete module over a continuous cis $\cis$, %down-complete idempotent semifield $\cis \neq \{0, 1\}$, 
and let $v : M \rightarrow \cis$. 
%If $\overline{M}$ denotes the normal completion of $M$, 
The following conditions are equivalent: 
\begin{enumerate}
	\item\label{litv1} $v$ is a smooth residuated form on $M$, 
	\item\label{litv2} $v$ is a non-degenerate continuous linear form on $M$, 
	\item\label{litv3} $v(\cdot) = \langle c, \cdot \rangle$, for some strongly archimedean element $c \in M$. %(\cdot \backslash c)^{-1}$. %, for all $x \in M$. %, where we write  $g \backslash c = 
	%\item there exists some $c \in M$ such that $v(x) = \bigwedge (\downarrow\!\! c \backslash x)$, for all $x \in M$. 
\end{enumerate}
If these conditions are satisfied, then $c$ is unique and equals the supremum of the set $\{ x\in M : 1 \geqslant v(x) \}$. 
\end{theorem}

%\textcolor{red}{Où utilise-t-on la continuité de K ?}

\begin{proof}
%As an introductory remark, note that, by completeness of $M$, the supremum of the nonempty set $\{ t \in \cis : c \geqslant x.t \}$ is reached, whenever this set is upper-bounded. This means that $c \geqslant x.(x \backslash c)$ or, in other words, that 
%\begin{equation}\label{eq:resid}
%c \geqslant x.t \Longleftrightarrow x \backslash c \geqslant t, 
%\end{equation}
%for all $t \in \cis$ and all $x, c \in M$ such that $x \backslash c \neq \top$. This implies that 
%\begin{equation}\label{eq:resid2}
%c.t \geqslant x \Longleftrightarrow (x \backslash c)^{-1} \leqslant t, 
%\end{equation}
%for all $t \in \cis$ and all $x, c \in M$. 
%
Equivalence between (\ref{litv1}) and (\ref{litv2}) is given by Lemma~\ref{lem:rescont}, and the implication (\ref{litv3}) $\Rightarrow$ (\ref{litv1}) was the purpose of Paragraph~\ref{ssec:archi}. 
So let us prove that (\ref{litv1}) implies (\ref{litv3}). Let $v$ be a smooth residuated form on $M$. Define $c$ as the supremum of the set $\{ x\in M : 1 \geqslant v(x) \}$, i.e.\ $c = v^{\#}(1)$, where $v^{\#}$ is the adjoint of $v$. 
If one notices that $1$ is a strongly archimedean element in $\cis$, then $c$ is a strongly archimedean element in $M$ by Lemma~\ref{lem:arch1}, and one has 
$$
\langle 1, v(x) \rangle = \langle v^{\#}(1), x \rangle, 
%(f(x) \backslash c) = (x \backslash f^{\#}(c)), 
$$
for all $x \in M$, that is 
$
v(x) = \langle c, x \rangle, 
$
for all $x \in M$. 
Uniqueness is deduced from the fact that $x \leqslant c \Leftrightarrow 1 \geqslant \langle c, x \rangle$. 
%The definition of $v$ gives $v(x) = \bigwedge \{t \in \cis : x \leqslant c.t \}$ for all $x \in M$, so that $v$ is a smooth linear form by Proposition~\ref{prop:nbm2}. 
%The complete $\cis$-module $M$ can be embedded in its ``one-point completion'' $\overline{M} = M \cup \{ \top \}$ (which is nothing but the Dedekind--MacNeille completion of $M$ exposed in Example~\ref{ex:dmc}). 
%Applying Lemma~\ref{lem:mf}, $v : x \mapsto (x \backslash c)^{-1}$, where $c$ is an archimedean element in $M$, is a non-degenerate continuous linear form on $M$.  	
%Conversely, let $v$ be a non-degenerate continuous linear form on $M$. Let $c$ be the supremum in $\overline{M}$ of the subset $\{ x \in M : 1 \geqslant v(x) \}$, and let us show first that $c \neq \top$. 
%LA CLEF DANS LITVINOV, c'est juste de voir que $\overline{M} = M \cup \{\top\}$ et que nécessairement $c \neq \top$. 
\end{proof}

\begin{example}[Example~\ref{ex:evt2} continued]\label{ex:evt3}
We introduced the linear form $v$ on $M_{\mu, \sigma, \xi}$ defined by $v(x) = (1 + \xi \frac{x - \mu}{\sigma})^{1/\xi}$ if $\xi$ is non-zero, $v(x) = \exp(\frac{x -\mu}{\sigma})$ otherwise. An easy computation shows that $\mu$ is the supremum of $\{ x \in M_{\mu, \sigma, \xi} : 1 \geqslant v(x) \}$, and  that $v(x) = \langle \mu, x \rangle$ %(x \backslash \mu)^{-1}$ 
for all $x \in M_{\mu, \sigma\, \xi}$. 
Moreover, $\mu$ is strongly archimedean (with respect to $\F = \UpperStar$), for if $t > s$, then $\mu.t = \mu + \sigma \frac{t^{\xi} - 1}{\xi} > \mu.s$ if $\xi$ is non-zero, and $\mu.t = \mu + \sigma \log(t) > \mu.s$ otherwise. 
Hence, $v$ is a smooth residuated form on $M_{\mu, \sigma, \xi}$ (where smoothness is understood with respect to $\F = \UpperStar$). 
\end{example}

This theorem upgrades a result by Cohen et al.\ \cite[Corollary~39]{Cohen04} deduced from a geometric Hahn--Banach type theorem \cite[Theorem~34]{Cohen04}. 
%These authors, in the algebraic framework of \textcolor{myred}{complete} modules over an idempotent semifield (or, more generally, over a \textit{reflexive} idempotent semiring, see the definition in \cite[§~4.3]{Cohen04}), showed that, if $M'$ denotes the set of continuous linear forms on $M$, then $(M, M')$ is a dual pair \cite[Corollary~37]{Cohen04} and every element $v$ of $M'$ can be represented as $v(x) = (x \backslash c)^{-1}$, for some $c \in M$ \cite[Corollary~39]{Cohen04}. 
%These results were deduced from a geometric Hahn--Banach type theorem \cite[Theorem~34]{Cohen04}. 
A different formulation will be proved in Section~\ref{sec:rmf}, in the framework of \textit{module extensions}. 
As for now, we use Theorem~\ref{thm:litv} to reprove the Radon--Nikodym theorem for the Shilkret integral. 

%\begin{corollary}
%Let $M$ be a complete module over a cis $\cis$. 
%With respect to the filter selection $\PFilters$, there is an isomorphism $\Phi$ between the module $M_a^{\op}$ of archimedean elements of $M$ and the module $M'$ of non-degenerate continuous linear forms on $M$, given by $\Phi(c) = (\cdot \backslash c)^{-1}$. 
%\end{corollary}

%\begin{proof}
%
%\end{proof}

%%%%%%%%%%%%%%%%%%%%%%%%%%%%%%%%%%%%%%%%%%%%%%%%
%%%%%%%%%%%%%%%%%%%%%%%%%%%%%%%%%%%%%%%%%%%%%%%%
%%%%%%%%%%%%%%%%%%%%%%%%%%%%%%%%%%%%%%%%%%%%%%%%
%%%%%%%%%%%%%%%%%%%%%%%%%%%%%%%%%%%%%%%%%%%%%%%%
\section{The Radon--Nikodym theorem: a different perspective}\label{sec:rn}

We come back to the Radon--Nikodym theorem for the Shilkret integral surveyed in \cite[Chapter~I]{Poncet11}; %Poncet \cite{??}; 
here we deduce this result from the order-theoretical developments of the previous section. 
See \cite[Chapters~I-II]{Poncet11} for definitions and notations related to maxitive measures. 
The filter selection used throughout this section is $\PFilters$, i.e.\ the one that selects principal filters.

\subsection{Complements on $\sigma$-complete modules}

The next result prepares applications to the Radon--Nikodym theorem. It gives sufficient conditions on a module $M$ over a cis $\cis$ in order that every linear form on $M$ be continuous. 
%A poset $P$ is \textit{principal} (resp. \textit{locally principal}) if every ideal (resp. every bounded ideal) of $P$ is principal. 
%A module $M$ is \textit{locally $\sigma$-complete} if every upper-bounded countable subset has a supremum and 
The module $M$ is \textit{$\sigma$-complete} %(resp.\ \textit{$\sigma$-complete}) 
if every upper-bounded countable %(resp.\ arbitrary) 
subset has a supremum and if $M$ is completable, i.e.\ if 
$$
x . \bigoplus T = \bigoplus_{t \in T} x.t, 
$$
for all $x \in M$, $T \subset \cis$ with supremum. 
We say that $M$ is \textit{$\sigma$-principal} if every upper-bounded $\sigma$-ideal is principal, i.e.\ of the form $\downarrow\!\! x$ for some $x \in M$. 
A subset $G$ of $M$ is \textit{generating} if, for all $x \in M$, $x = \bigoplus \downarrow\!\! x \cap G$. Also, the module is \textit{countably generated}  %(resp.\ \textit{countably generated}) 
if there exists a generating subset %(resp.\ a generating countable subset) 
$G$ such that $\downarrow\!\! x \cap G$ is countable, for all $x \in M$. 
A linear form $v$ on $M$ is \textit{$\sigma$-continuous} if, for every countable subset $X \subset M$ admitting a supremum in $M$, $v(X)$ has a supremum in $\cis$ satisfying $v(\bigoplus X) = \bigoplus v(X)$. 

\begin{proposition}\label{principalcomplet}
Let $M$ be a %(locally) 
$\sigma$-complete module over a cis $\cis$. 
\begin{enumerate}
	\item\label{pal1} If $M$ is %(locally) 
	countably generated, then $M$ is %(locally) 
	$\sigma$-principal. %admet une famille génératrice finie (resp. dénombrable) alors $X$ est principal. (resp.  $\sigma$-principal). 
	\item\label{pal2} If $M$ is %(locally) 
	$\sigma$-principal, then $M$ is %(locally) 
	complete and every $\sigma$-continuous linear form is continuous. %Si $X$ est principal (resp. $\sigma$-principal), alors $X$ est complet et tout caractère (resp. tout $\sigma$-caractère) est continu.
%	\item\label{pal3} Conversely, if $M$ is (locally) complete and every linear form is continuous, then $M$ is (locally) principal. %Réciproquement, si $X$ est complet et tout caractère (resp. tout $\sigma$-caractère) est continu, alors $X$ est principal (resp. $\sigma$-principal).
\end{enumerate}% En particulier, si $(E,\mathrsfs{E})$ est un espace semi-mesurable muni d'une mesure $\sigma$-maxitive $\tau$, 
\end{proposition}

\begin{proof}
(\ref{pal1}) Assume that $M$ is %locally 
countably generated by some subset $G$. Let $I$ be an upper-bounded $\sigma$-ideal of $M$. If $u$ is an upper-bound of $I$, the subset $I \cap G$ is included in the countable subset $\downarrow\!\! u \cap G$, hence is countable. So let $x := \bigoplus G \cap I \in I$. It is easily seen that $x = \bigoplus I$, hence $I = \downarrow\!\! x$, i.e.\ $I$ is a principal ideal. 

(\ref{pal2}) Assume that $M$ is %locally 
$\sigma$-principal, and let $X$ be an upper-bounded subset of $M$. The $\sigma$-ideal $I$ generated by $X$ is made up of elements lower than joins of countable subsets of $X$. Since $I$ is upper-bounded, it is principal, so we have $I = \downarrow \!\!x$ for some $x \in I$, and $x$ is of the form $x = \bigoplus G$ for some countable subset $G$ of $X$. Thus, we have $x = \bigoplus G = \bigoplus I = \bigoplus X$, so that $M$ is %locally 
complete. Moreover, if $v : M \rightarrow \cis$ is a $\sigma$-continuous linear form, then $v(\bigoplus X) = v(\bigoplus G) = \bigoplus v(G) \leqslant \bigoplus v(X)$, hence $v$ is continuous. 
%
%$\bullet$ Soit $(x_j)_{j\in J}$ une famille quelconque d'éléments de $X$ et soit $I = \{ x \in X : \exists J'\subset J \mbox{ dénombrable}, x \preccurlyeq \bigoplus_{j\in J'} x_j \}$ le $\sigma$-idéal de $X$ engendré par les $x_j$. Soit $\ell \in I$ tel que $I = \downarrow \! \ell$, que l'on peut choisir sous la forme $\ell = \bigoplus_{j\in J'} x_j$. Il est clair que $\ell$ est la borne supérieure de $(x_j)_{j\in J}$. De plus si $v : X \rightarrow \sci$ est un $\sigma$-caractère, $v(\bigoplus_{j\in J} x_j) = v(\ell) = \bigoplus_{j\in J'} v(x_j) \leqslant \bigoplus_{j\in J} v(x_j)$, d'où la continuité de $v$.
%
%(\ref{pal3}) Let $I$ be an upper-bounded ideal of $M$, and let $a := \bigoplus I \in M$. The map $v : M \rightarrow \cis$ defined by $v(x) = 1_{I^c}(x)$ is a linear form, hence is continuous. Thus, $v(a) = \bigoplus v(I) = 0$, so that $a \in I$. This implies that $I$ is principal (equal to $\downarrow \! a$).
\end{proof}

\subsection{Maxitive measures as linear forms}

Let $\mathrsfs{E}$ be a semi-$\sigma$-algebra on some nonempty set $E$ and $\nu, \tau$ be $\sigma$-maxitive measures on $\mathrsfs{E}$. We shall assume that $\nu$ is finite and absolutely continuous with respect to $\tau$, in symbols $\nu \absc \tau$. 
In order to apply the results of Section~\ref{sec:continuous}, we merely want to get rid of the collection of $\tau$-negligible subsets. % $\mathrsfs{N}_\tau$. 
We could consider the quotient space $\mathrsfs{E}/\tau$, but this would not give us the structure of module over the idempotent semifield $\scit = (\scif, \max, \times)$ that we need. A better idea is the following. %then to consider the quotient space $\mathbf{M} = L^{1}_+$, precisely defined as follows. 

Let $\mathrsfs{E}_+^{\nu} = \mathrsfs{L}^1_+(E, \mathrsfs{E}, \nu)$ be the set of all $\nu$-integrable lsm maps $g : E \rightarrow \sci$. A map $n$ in $\mathrsfs{E}_+^{\nu}$ is \textit{$\tau$-negligible} if the subset $\{ n > 0 \}$ is $\tau$-negligible. %telles que $m(b) < \infty$.
We define on $\mathrsfs{E}_+^{\nu}$ the equivalence relation $\langle\rangle$ 
%$\lozenge$ 
by $f \langle\rangle g$ if and only if, for some $\tau$-negligible map $n$, we have $f \oplus n = g \oplus n$. We denote by $\langle g \rangle$ the equivalence class of a $g \in \mathrsfs{E}_+^{\nu}$. %\textcolor{myred}{REMPLACER PAR $\mathrsfs{E}_+^{\nu}/\tau$ ! }
Then the quotient set $\mathbf{M} := \mathrsfs{E}_+^{\nu}/\tau := \mathrsfs{E}_+^{\nu}/\langle\rangle$ is a $\sigma$-complete module over $\scit$ with external multiplication $\mathbf{f}.t := \langle t.f \rangle$ and countable addition $\bigoplus_{j = 1}^\infty \mathbf{g}_j = \langle \bigoplus_{j=1}^\infty g_j \rangle$, %$\mathbf{f} \oplus \mathbf{g} = \langle f \oplus g \rangle$, 
for all $t \in \scif$ and $\mathbf{f} = \langle f \rangle, \mathbf{g}_j = \langle g_j \rangle \in \mathbf{M}$. The induced partial order is $\mathbf{f} \leqslant \mathbf{g}$ if and only if $\{ f > g \}$ is $\tau$-negligible. The reader can check that the previous definitions do not depend on the choice of the representatives $f$, $g$, etc. 

Recall that $\tau$ on $\mathrsfs{E}$ is \textit{localizable} if, for each $\sigma$-ideal $\mathrsfs{I}$ of $\mathrsfs{E}$, there exists some $L \in \mathrsfs{E}$ such that 
\begin{itemize}
	\item $S \setminus L$ is $\tau$-negligible, for all $S \in \mathrsfs{I}$, 
	\item if there is some $G \in \mathrsfs{E}$ such that $S \setminus G$ is $\tau$-negligible for all $S \in \mathrsfs{I}$, then $L \setminus G$ is $\tau$-negligible. 
\end{itemize}
In this case, $\mathrsfs{I}$ is said to be \textit{localized} in $L$. 

\begin{proposition}\label{prop:loc1}
Let $\nu, \tau$ be $\sigma$-maxitive measures on $\mathrsfs{E}$. Assume that $\nu$ is finite and such that $\nu \absc \tau$. %a $\sigma$-algebra $\mathrsfs{B}$. 
Then $\tau$ is localizable (resp.\ $\sigma$-principal) if and only if $\mathrsfs{E}_+^{\nu}/\tau$ is a complete module (resp.\ a $\sigma$-principal module). %$\mathrsfs{B}_+^{\tau}/\tau$ is locally-complete. 
\end{proposition}

\begin{proof}
A preliminary remark is that, since $\nu$ is finite, the lsm map $1_G$ is $\nu$-integrable for all $G \in \mathrsfs{E}$. 
Assume that $\mathbf{M} = \mathrsfs{E}_+^{\nu}/\tau$ is a complete module, and let $\mathrsfs{I}$ be a $\sigma$-ideal of $\mathrsfs{E}$. Then the $\sigma$-ideal $I$ generated by $\{ \langle 1_S \rangle : S \in \mathrsfs{I} \}$ is upper-bounded (by $\langle 1_E \rangle$) in $\mathbf{M}$. Hence there is some $f \in \mathrsfs{E}_{+}^{\nu}$ such that $\langle f \rangle$ is the supremum of $I$. In particular, if $S \in \mathrsfs{I}$, there is some $\tau$-negligible lsm map $n$ such that $1_S \leqslant f \oplus n$, so that $S \subset L \cup \{ n > 0 \}$, where $L := \{ f > 2^{-1} \}$. As a consequence, $S \setminus L$ is $\tau$-negligible for all $S \in \mathrsfs{I}$. To show that $\mathrsfs{I}$ is localized in $L$, let $G \in \mathrsfs{E}$ such that $S \setminus G$ is $\tau$-negligible for all $S \in \mathrsfs{I}$. 
Then $\langle 1_G \rangle$ is an upper-bound of $I$, so that $\langle f \rangle \leqslant \langle 1_G \rangle$ by definition of $f$. Since $2^{-1}.1_L \leqslant f$, we deduce that $L \setminus G$ is $\tau$-negligible, hence that $\tau$ is localizable. 

If $\mathbf{M}$ is %locally 
$\sigma$-principal, we can impose $L$ to belong to $\mathrsfs{I}$ and to be such that $\langle 1_L \rangle$ generates $I$. Then $L$ generates $\mathrsfs{I}$, and this proves that $\tau$ is $\sigma$-principal. 
%, for all $G \in \mathrsfs{I}$, $\langle 1_G \rangle$ is less than $\langle 1_L \rangle$ in $\mathbf{M}$, i.e.\ there exists some $n \in \mathrsfs{E}_{+}^{\tau}$ such that $\tau(n > 0) = 0$ and $1_G \leqslant 1_L \oplus n$. This implies $G \subset L \cup N$, where $N$ is the $\tau$-negligible subset $\{ n > 0 \}$. This proves that $\tau$ is $\sigma$-principal. 

Conversely, suppose that $\tau$ is localizable, and let $I$ be an upper-bounded $\sigma$-ideal of $\mathbf{M}$. If $q \in \mathbb{Q}_+$, let $\mathrsfs{I}_q = \{ \{f>q\} : \langle f \rangle \in I \}$. This is a $\sigma$-ideal, hence it is localized in some $L_q \in \mathrsfs{E}$. 
Let $\langle g \rangle$ be an upper-bound of $I$. Then $S \setminus \{ g > q \}$ is $\tau$-negligible, for all $S \in \mathrsfs{I}_q$ and all $q \in \mathbb{Q}_+$. Since $\mathrsfs{I}_q$ is localized in $L_q$ we deduce that $L_q \setminus \{ g > q\}$ is $\tau$-negligible. This implies that the map $\ell$ defined by $\ell = \bigoplus_{q \in \mathbb{Q}_+} q.1_{L_q}$ is $\nu$-integrable and satisfies $\langle \ell \rangle \leqslant \langle g \rangle$. 
To show that $\langle \ell \rangle$ is the supremum of $I$, it suffices to prove that $\langle \ell \rangle$ is an upper-bound of $I$. 
If $\langle f \rangle \in I$, there exists some $\tau$-negligible subset $N_q \in \mathrsfs{E}$ such that $\{ f > q \} \subset L_q \cup N_q$. If $n = \bigoplus_{q \in \mathbb{Q}} q.1_{N_q}$, then $\{ n > 0 \} \subset \bigcup_{q \in \mathbb{Q}_+} N_q$, so $n$ is $\tau$-negligible. We have $f \leqslant \ell \oplus n$, so that $\langle f \rangle \leqslant \langle \ell \rangle$. This proves that $\langle \ell \rangle$ is the supremum of $I$, and that $\mathbf{M}$ is complete. 

If $\tau$ is $\sigma$-principal, then the set $L_q$ can be choosen of the form $\{ \ell_q > q \}$, where $\langle \ell_q \rangle \in I$. It can be seen that $\ell$ and $\bigoplus_{q \in \mathbb{Q}_+} \ell_q$ are equivalent, so that $\langle \ell \rangle = \langle \bigoplus_{q \in \mathbb{Q}_+} \ell_q \rangle \in I$. This shows that $I$ is principal, so that $\mathbf{M}$ is a $\sigma$-principal module. 
%, and let $I$ be an upper-bounded $\sigma$-ideal of $\mathbf{M}$. If $q \in \mathbb{Q}_+$, let $\mathrsfs{I}_q = \{ \{f>q\} : \langle f \rangle \in I \}$. This is a $\sigma$-ideal, hence there is some $\langle \ell_q \rangle \in I$ such that 
%Hence, $N_q = \{f>q\} \setminus \{\ell_q>q\}$ is $\tau$-negligible, for all $\langle f \rangle \in I$. As a consequence, $f \leqslant \ell \oplus n$, where $\ell = \bigoplus_{q \in \mathbb{Q}_+} \ell_q$ and $n = \bigoplus_{q \in \mathbb{Q}_+} q. 1_{N_q}$. Since $\{ n > 0 \} \subset \bigcup_{q \in \mathbb{Q}_+} N_q$, the subset $\{n > 0 \}$ is $\tau$-negligible. Moreover, $\langle \ell \rangle \in I$ because $I$ is upper-bounded, so that $\langle f \rangle$ is less than $\langle \ell \rangle$, for all $\langle f \rangle \in I$, and we have proved that $\mathbf{M}$ is $\sigma$-principal. 
% such that  $N_q = \{f>q\} \setminus \{\ell_q>q\}$ is $\tau$-negligible, for all $\langle f \rangle \in I$. As a consequence, $f \leqslant \ell \oplus n$, where $\ell = \bigoplus_{q \in \mathbb{Q}_+} \ell_q$ and $n = \bigoplus_{q \in \mathbb{Q}_+} q 1_{N_q}$. Since $\{ n > 0 \} \subset \bigcup_{q \in \mathbb{Q}_+} N_q$, the subset $\{n > 0 \}$ is $\tau$-negligible. Moreover, $\langle \ell \rangle \in I$ because $I$ is upper-bounded, so that $\langle f \rangle$ is less than $\langle \ell \rangle$, for all $\langle f \rangle \in I$, and we have proved that $\mathbf{M}$ is $\sigma$-principal. 
\end{proof}

We denote by $\mathbf{v}$ the map induced by $\nu$ on $\mathbf{M}$, i.e.\ 
$$
\mathbf{v}(\mathbf{f}) = \ddint f \, d\nu, 
$$
for all $\mathbf{f} = \langle f \rangle \in \mathbf{M}$. Since $\mathrsfs{E}_+^{\nu}$ demands $\nu$-integrable maps, we have $\mathbf{v}(\mathbf{f}) < \infty$, so $\mathbf{v}$ is a $\sigma$-continuous linear form on $\mathbf{M}$. 
%which is a linear form. %$\sigma$-caractère.
We shall say that $\nu$ is \textit{$\tau$-continuous} if $\mathbf{v}$ is continuous. % linear form on $\mathbf{M}$. 
% (see the definition at the beginning of Section~\ref{sec:conmaxmaps}). 
 %, i.e.\ si pour toute famille $(\mathbf{b}_i)_{i\in I}$ d'éléments de $\mathbf{X}$ admettant une borne supérieure, on a $\mathbf{v}(\bigoplus_{i\in I} \mathbf{b}_i) = \bigoplus_{i\in I}  \mathbf{v}(\mathbf{b}_i)$. 
%We say that $\tau$ is \textit{localizable} if $\mathbf{M}$ is a complete module (see Maharam \cite{Maharam47} for the original definition). 
As a corollary of Theorem~\ref{thm:litv} we have the following result. Recall that a map $g : E \rightarrow \sci$ is \textit{upper-semimeasurable} or \textit{usm} if $\{ g < t \} \in \mathrsfs{E}$ for all $t \in \scif$. 

\begin{theorem}\label{thm:usm}
Let $\nu$, $\tau$ be $\sigma$-maxitive measures on $\mathrsfs{E}$. Assume that $\nu$ is finite and $\tau$ is localizable. 
%If $\tau$ is localizable, %On suppose que $\tau$ est de Maharam et que $\nu \absc \tau$. 
Then the following assertions are equivalent: 
\begin{itemize}
	\item $\nu \absc \tau$ and $\nu$ is %\textcolor{myred}{non-degenerate and} 
	$\tau$-continuous, % Si $\nu$ admet une densité par rapport à $\delta_\tau$ alors $\nu \absc \tau$ et $\nu$ est $\tau$-continue, 
	\item $\nu$ has a usm relative density with respect to $\tau$. %Si $\nu \absc \tau$ et si $\nu$ est $\tau$-continue, alors $\nu$ admet une densité par rapport à $\delta_\tau$. \\
\end{itemize}
%Moreover, if $\tau$ is principal, then $\tau$ is Maharam, and $\nu \absc \tau$ implies that $\nu$ be $\tau$-continuous. 
%Alors $\nu$ admet une densité par rapport à $\delta_\tau$ si et seulement si $\nu \absc \tau$ et $\nu$ est $\tau$-continue. De plus, ces conditions sont vérifiées si $\tau$ est localisable et $\nu \absc \tau$.
\end{theorem}

%\textcolor{myblue}{ATTENTION usm ici veut dire que $\{ g < t \} \in \mathrsfs{E}$ et non pas que $\{ g \geqslant t \} \in \mathrsfs{E}$ ! }

\begin{proof}
Since $\nu$ is finite and $\tau$ is localizable, $\mathbf{M} =	\mathrsfs{E}_+^{\nu}/\tau$ is a complete module by Proposition~\ref{prop:loc1}. 
From the identity $\langle \mathbf{c}, \mathbf{f} \rangle %(\mathbf{f}\backslash\mathbf{c})^{-1} 
= \bigoplus_{x\in E}^{\tau} \frac{f(x)}{c(x)}$, which holds for all $\mathbf{f} = \langle f \rangle, \mathbf{c} = \langle c \rangle \in \mathbf{M}$, we deduce that $\nu$ has a usm relative density with respect to $\tau$ if and only if there is some $\mathbf{c} \in \mathbf{M}$ such that $\mathbf{v}(\cdot) = \langle \mathbf{c}, \cdot \rangle$. %(\cdot\backslash\mathbf{c})^{-1}$. 
This situation implies that $\nu \absc \tau$ and $\nu$ is $\tau$-continuous by Theorem~\ref{thm:litv}. 

For the converse statement, assume that $\nu \absc \tau$ and $\nu$ is $\tau$-continuous. We only have to prove that $\mathbf{v}$ is non-degenerate, for then Theorem~\ref{thm:litv} gives the desired result. 
So let $f \in \mathrsfs{E}_+^{\nu}$ such that $\mathbf{v}(\mathbf{f}) \leqslant 1$, where $\mathbf{f} = \langle f \rangle$. Then, for all rational numbers $q > 0$, the subset $\{f > q \}$ is in the $\sigma$-ideal $\mathrsfs{I}_q = \{ G \in \mathrsfs{E} : \nu(G) \leqslant q^{-1} \}$. 
Since $\tau$ is localizable, $\mathrsfs{I}_q$ is localized in some $L_q \in \mathrsfs{E}$, 
and since $\nu$ is $\tau$-continuous, we have $\nu(L_q) \leqslant q^{-1}$. 
As a consequence, the map $g = \bigoplus_{q \in \mathbb{Q}_+} q.1_{L_q}$ is lsm, $\nu$-integrable (with $\ddint g \, d\nu \leqslant 1$), and such that $\langle f \rangle \leqslant \langle g \rangle$. This proves that the subset $\{ \mathbf{f} \in \mathbf{M} : 1 \geqslant \mathbf{v}(\mathbf{f}) \}$ is upper-bounded in $\mathbf{M}$, i.e.\ that $\mathbf{v}$ is non-degenerate. 
%The assertions are equivalent thanks to Theorem~\ref{thm:litv} and the fact that  
\end{proof}

%\textcolor{myred}{Le corollaire suivant, manifestement faux, pose problème... }

\begin{corollary}\label{coro:usm}
Let $\nu$, $\tau$ be $\sigma$-maxitive measures on $\mathrsfs{E}$. Assume that $\tau$ is %$\sigma$-finite and 
$\sigma$-principal. %On suppose que $\tau$ est de Maharam et que $\nu \absc \tau$.
Then $\nu \absc \tau$ if and only if  $\nu$ has a usm relative density with respect to $\tau$. %the following assertions are equivalent: 
%\begin{itemize}
%	\item $\nu \absc \tau$, 
%	\item $\nu$ has a usm relative density with respect to $\tau$. 
%\end{itemize}
\end{corollary}

\begin{proof}
We first suppose that $\nu$ is finite. 
If $\tau$ is $\sigma$-principal, then $\mathrsfs{E}_+^{\nu}/\tau$ is %locally 
a $\sigma$-principal module by Proposition~\ref{prop:loc1}. By Proposition~\ref{principalcomplet}, $\nu$ is $\tau$-continuous and $\tau$ is localizable. So Theorem~\ref{thm:usm} applies, and $\nu$ admits a usm relative density with respect to $\tau$. 
%, in particular a complete module by Proposition~\ref{principalcomplet}. This implies that $\tau$ is localizable and, still by Proposition~\ref{principalcomplet}, that $\nu$ is $\tau$-continuous, for all $\nu \absc \tau$. To conclude, use Theorem~\ref{thm:usm}. 
In the case where $\nu$ is non-finite, we replace $\nu$ by $\nu_1 : B \mapsto \arctan \nu(B)$, which has a usm relative density $c_1$ with respect to $\tau$. Therefore, $\tan c_1$ is a usm relative density of $\nu$ with respect to $\tau$. 
%If $\tau$ is principal, then $\tau$ is Maharam, and $\nu \absc \tau$ implies that $\nu$ be $\tau$-continuous. 
\end{proof}

%\begin{corollary}
%Let $\nu$ be a $\sigma$-maxitive measure on $\mathrsfs{E}$. 
%Then $\nu$ is $\nu$-continuous if and only if $\nu$ is autocontinuous, in the sense that $\nu(\cdot) = \bigoplus_{x \in \cdot}^{\nu} c(x)$, for some usm map $c$. %, that we called autocontinuity property in Propositions~\ref{??}. 
%Moreover, if $\nu$ is $\sigma$-principal, then these conditions are satisfied. 
%\end{corollary}

%\begin{proof}
%Use Theorem~\ref{thm:usm} and Corollary~\ref{coro:usm} with $\tau = \nu$. 
%\end{proof}

\begin{example}
Let $E$ be a topological space, $\mathrsfs{E}$ be the collection $\mathrsfs{G}$ of open subsets of $E$, and $\tau = \delta_{\#}$. %, where $\delta_{\#}$ is the maxitive measure defined by $\delta_{\#}(G) = 1$ if $G \neq \emptyset$ and $\delta_{\#}(\emptyset) = 0$. 
Then $\delta_{\#}$ is localizable and $\nu \absc \delta_{\#}$, for all maxitive measures $\nu$ on $\mathrsfs{G}$. Moreover, $\nu$ is $\delta_{\#}$-continuous if and only if $\nu$ is completely maxitive if and only if $\nu$ has a usc cardinal density. %, thus one  on retrouve donc ainsi le théorème~\ref{caspart}. 
Also, $\delta_{\#}$ is $\sigma$-principal if and only if every subset of $E$ is Lindel\"of (then $E$ is usually said to be \textit{hereditarily Lindel\"of}, a property that is implied by second-countability), in which case every $\nu$ on $\mathrsfs{G}$ has a cardinal density. 
\end{example}

To conclude this section we propose a new proof of the Sugeno--Muro\-fushi theorem, which is a Radon--Nikodym like theorem for the Shilkret integral (see \cite[Theorem~I-6.4]{Poncet11}). % (see \cite[Corollary~8.4]{Sugeno87}; see also [Chapter~I, Theorem~5.4]). 
% Poncet \cite[Theorem~??]{Poncet??}). %, which expresses as a Radon--Nikodym theorem with respect to the Shilkret integral. 
%whose terms are first reminded for ease of reading. 

%\begin{theorem}[Sugeno--Murofushi's theorem]
%Let $\nu$, $\tau$ be $\sigma$-maxitive measures on $\mathrsfs{B}$. Assume that $\tau$ is localizable. Then $\nu \absc \tau$ if and only if there exists some $\mathrsfs{B}$-measurable map $c : E \rightarrow \sci$ such that 
%\begin{eqnarray*}
%\nu(B) = \dint{B} c\, d\tau, 
%\end{eqnarray*}
%for all $B\in \mathrsfs{B}$. 
%If these conditions are satisfied, then $c$ is unique $\tau$-almost everywhere. Moreover, if $\tau$ is semifinite, one can choose a map $c$ taking only finite values. 
%\end{theorem}

\begin{theorem}[Sugeno--Murofushi]\label{sugeno-murofushi2}
Let $\nu$, $\tau$ be $\sigma$-maxitive measures on a $\sigma$-algebra $\mathrsfs{B}$. 
Assume that $\tau$ is $\sigma$-finite and $\sigma$-principal. 
%Assume that $\tau$ is semifinite and principal. 
Then $\nu \absc \tau$ if and only if there exists some $\mathrsfs{B}$-measurable map $c : E \rightarrow \sci$ such that 
$$
\nu(B) = \dint{B} c\, d\tau, 
$$
for all $B\in \mathrsfs{B}$. 
If these conditions are satisfied, then $c$ is unique $\tau$-almost everywhere. %Moreover, if $\tau$ is semifinite, one can choose a map $c$ taking only finite values. 
\end{theorem}

\begin{proof}%[Proof of Theorem~\ref{sugeno-murofushi}]
%First assume that $\tau$ is finite. 
If $\nu \absc \tau$, then by Corollary~\ref{coro:usm} 
%there exists $\mathrsfs{B}$-measurable maps $c_1, c_2 : E \rightarrow \sci$ such that $\nu(B) = \bigoplus^{\tau}_{x\in B} c_1(x)$ and  $\tau(B) = \bigoplus^{\tau}_{x\in B} c_2(x)$, for all $B \in \mathrsfs{B}$. Since $\tau$ is semifinite, one can choose a map $c_2$ which takes only finite values. It is now easy to show that $\nu(\cdot) = \dint{\cdot} c \, d\tau$ on $\mathrsfs{B}$, where $c$ is defined by $c(x) = c_1(x)/c_2(x)$ if $c_2(x) \neq 0$, $c(x) = 0$ otherwise. 
there are $\mathrsfs{B}$-measurable maps $c_1, c_2 : E \rightarrow \sci$ such that $\nu(B) = \bigoplus^{\tau}_{x\in B} c_1(x)$ and  $\tau(B) = \bigoplus^{\tau}_{x\in B} c_2(x)$, for all $B \in \mathrsfs{B}$. Since $\tau$ is $\sigma$-finite, one can choose a map $c_2$ that takes only finite values (see \cite[Proposition~I-6.1]{Poncet11}). Using the fact that $\tau(\{ c_2 = 0 \}) = 0$, it is easy to show that $\nu(\cdot) = \dint{\cdot} c \, d\tau$ on $\mathrsfs{B}$, where $c$ is the measurable map defined by $c(x) = c_1(x)/c_2(x)$ if $c_2(x) \neq 0$, $c(x) = 0$ otherwise. 
\end{proof}

\section{Completable modules and the normal completion}\label{sec:dmcc}

\subsection{The normal completion of a completable module}

The following theorem defines the concept of \textit{normal completion} of a completable module, alias \textit{Dedekind--MacNeille completion} or \textit{completion by cuts}. See e.g.\ Ern\'e \cite{Erne91c} for the normal completion of quasiordered sets. 

\begin{theorem}\label{thm:nc}
Let $\cis$ be an idempotent semifield. A $\cis$-module is completable if and only if it can be continuously embedded into a complete $\cis$-module. %, i.e.\ if there exists a complete $\cis$-module $\overline{M}$ and a continuous injective morphism $i_M : M \rightarrow \overline{M}$.  
%$$
%x.\bigoplus T = \bigoplus (x.T), 
%$$
%for all $x \in M$ and all nonempty subsets $T$ of $\cis$. 
\end{theorem}

\begin{proof}[Sketch of the proof]
Sufficiency is obvious. For necessity, let $M$ be a completable $\cis$-module. 
We follow the usual Dedekind--MacNeille completion method for partially ordered sets. If $X \subset M$, we write $X^{\downarrow}$ (resp.\ $X^{\uparrow}$) for the subset of lower (resp.\ upper) bounds of $X$ in $M$, and we write $X^{\uparrow\downarrow}$ instead of $(X^{\uparrow})^{\downarrow}$. A subset $X$ of $M$ is \textit{closed} if $X^{\uparrow\downarrow} = X$, and \textit{proper} if either $X \neq M$ or $M$ has a greatest element. 
Let $\mathrsfs{N}(M)$ be the collection of all proper closed subsets $X$ of $M$. If $X \oplus X' := (X \cup X')^{\uparrow\downarrow}$ for all proper closed subsets $X, X'$, then $X \oplus X'$ is closed, proper (to prove this, note that a closed subset is proper if and only if it is upper-bounded) and $(\mathrsfs{N}(M), \oplus, \{0 \})$ is a commutative idempotent monoid. The partial order induced by $\oplus$ on $\mathrsfs{N}(M)$ is the inclusion, i.e.\ $X \leqslant X' \Leftrightarrow X \subset X'$. 
For the external multiplication we let $X.t := \{ x.t : x \in X \}$ if $t \neq 0$ and $X.0 = \{ 0 \}$, and one can check that $X.t$ is proper closed for all proper closed subsets $X$. 
Also, since $M$ is completable, the following relations hold:  
\begin{align*}
(X \oplus X').t &= X.t \oplus X'.t, \\ 
X.\bigoplus T &= \bigoplus_{t \in T} X.t, 
\end{align*}
for all $X, X' \in \mathrsfs{N}(M)$, $t \in \cis$ and $T \subset \cis$ with supremum, and 
$$
X.\bigwedge F = \bigwedge_{f \in F} X.f, 
$$
for all $X \in \mathrsfs{N}(M)$ and $\F$-sets $F$ in $\cis$ with infimum.
Thus, $\mathrsfs{N}(M)$ is a completable $\cis$-module, which is actually complete for  
$$
\bigoplus_{j\in J} X_j = (\bigcup_{j\in J} X_j)^{\uparrow\downarrow}, 
$$ 
for all upper-bounded families $(X_j)_{j \in J}$ of proper closed subsets. Note that the infimum in $\mathrsfs{N}(M)$ satisfies  
$$
\bigwedge_{j\in J} X_j = \bigcap_{j\in J} X_j, 
$$ 
for all families $(X_j)_{j \in J}$ of proper closed subsets. 

To embed $M$ into $\mathrsfs{N}(M)$, let $i_M : M  \rightarrow \mathrsfs{N}(M), x \mapsto \downarrow\!\! x$. This map $i_M$ is well defined, for $\downarrow\!\! x$ is proper closed for all $x \in M$. Clearly, we have $i_M(x.t) = i_M(x).t$ and $i_M(x \oplus y) = i_M(x) \oplus i_M(y)$ for all $x, y \in M$ and $t \in \cis$, so that $i_M$ is an injective morphism. Moreover, for all subsets $X$ of $M$ with supremum (resp.\ with infimum), we have $i_M(\bigoplus X) = \bigoplus i_M(X)$ (resp.\ $i_M(\bigwedge X) = \bigwedge i_M(X)$), so that $i_M$ is continuous. 
\end{proof}

\begin{remark}
\textcolor{mygreen}{If $\F = \PFilters$, then a module is completable if and only if it is \textit{b-regular} in the sense of Litvinov et al.\ \cite[Definition~3.9]{Litvinov01}.}
\end{remark}

\begin{remark}\label{rk:ub}
Identifying $M$ and $i_M(M)$, every element of $\mathrsfs{N}(M)$ can be expressed as a supremum (resp.\ an infimum) of elements of $M$. In particular, every element of $\mathrsfs{N}(M)$ is upper-bounded by some element of $M$. 
\end{remark}

\begin{remark}
Every idempotent semifield considered as a module over itself is completable. However, an idempotent semifield can be embedded into a complete idempotent semifield if and only if it is commutative (see Remark~\ref{rk:iwa}). 
\end{remark}

\subsection{Cut-stability and extensions}

In this paragraph we give two categorical results on the normal completion. Since every $\{0,1\}$-module is completable, they extend that of Ern\'e \cite{Erne91c}. 

Ern\'e introduced the concept of cut-stability, that we modify as follows. A map $f : M \rightarrow N$ is \textit{lower cut-stable} if  
$$
f(X^{\uparrow})^{\downarrow} = f(X)^{\uparrow\downarrow},  
$$
for all %nonempty 
subsets $X$ of $M$, and \textit{cut-stable} if it is lower cut-stable and such that 
$$
f(F^{\downarrow})^{\uparrow} = f(F)^{\downarrow\uparrow}, 
$$
for all $\F$-sets $F$ of $M$. 
For instance, the map $i_M$ that embeds a completable module into its normal completion (see the proof of Theorem~\ref{thm:nc}) is cut-stable. Note that every cut-stable morphism is continuous. 

%\begin{figure}
%	\centering
%		\includegraphics[width=0.99\textwidth]{C:/Maths/Article_RN/Cut-stable-maps-Erne91c.PNG}
%	\caption{Properties of maps between quasiordered sets. Source: \cite[Diagram~2.7]{Erne91c}}
%	\label{fig:Cut-stable-maps-Erne91c}
%\end{figure}

\begin{proposition}[Compare with \protect{\cite[Theorem~3.1]{Erne91c}}]
Let $\cis$ be an idempotent semifield. 
A morphism $f$ between completable $\cis$-modules $M$ and $N$ is lower cut-stable if and only if there exists a (unique) morphism $\mathrsfs{N}(f)$ between the normal completions $\mathrsfs{N}(M)$ and $\mathrsfs{N}(N)$ that preserves arbitrary existing suprema and extending $f$, i.e.\ such that the following diagram commutes: 
%The normal completion extends to a functor $\mathrsfs{N}$ on the category of completable $\cis$-modules with continuous morphisms. Moreover, for all continuous morphisms $f : M \rightarrow N$ between completable $\cis$-modules $M, N$, the following diagram commutes: 
\begin{diagram}
M            & \rTo^{f}       & N            \\
\dTo<{i_M}   &                & \dTo>{i_N}   \\
\mathrsfs{N}(M) & \rTo^{\mathrsfs{N}(f)} & \mathrsfs{N}(N) \\
\end{diagram}
Moreover, in the cases $\F = \UpperStar$ and $\F = \PFilters$, the morphism $f$ is cut-stable if and only if $\mathrsfs{N}(f)$ is continuous, and the normal completion extends to a functor $\mathrsfs{N}$ on the category of completable $\cis$-modules with continuous morphisms. 
\end{proposition}

\begin{proof}
Analogous to that of \cite[Theorem~3.1]{Erne91c}. 
\end{proof}

\begin{remark}
\textcolor{myblue}{In the case $\F = \Filters$, it is only possible to say that $\mathrsfs{N}(f)$ preserves infima of filters of $\mathrsfs{N}(M)$ of the form $\mathrsfs{F}_F = \{ X \in \mathrsfs{N}(M) : X \cap F \neq \emptyset\}$, with $F$ a filter of $M$. This is not enough to make $\mathrsfs{N}(f)$ smooth (i.e.\ continuous) in general.}
%s $\F = \UpperStar$ and $\F = \PFilters$, the normal completion extends to a functor $\mathrsfs{N}$ on the category of completable $\cis$-modules with continuous morphisms. Unfortunately, this does not seem to hold in the case $\F = \Filters$. 
\end{remark}

The following universal property of the normal completion is deduced immediately. 

\begin{corollary}[Compare with \protect{\cite[Corollary~3.2]{Erne91c}}]\label{coro:univ}
Let $\cis$ be an idempotent semifield. 
A morphism $f$ from a completable $\cis$-module $M$ into a complete $\cis$-module $N$ is 
lower cut-stable if and only if there exists a (unique) morphism from $\mathrsfs{N}(M)$ into $N$ arbitrary existing suprema and extending $f$, i.e.\ such that the following diagram commutes: 
%The following universal property holds: if $M$ is a completable $\cis$-module and $f : M \rightarrow N$ is a continuous morphism from $M$ to some complete $\cis$-module $N$, then there exists a unique continuous morphism $\mathrsfs{N}(f) : \mathrsfs{N}(M) \rightarrow N$ such that the following diagram commutes: 
\begin{diagram}
M            & \rTo_{f}            & N  \\
\dTo<{i_M}     & \ruDashto_{\mathrsfs{N}(f)} &    \\
\mathrsfs{N}(M) &                     &    \\
\end{diagram}
Moreover, in the cases $\F = \UpperStar$ and $\F = \PFilters$, the morphism $f$ is cut-stable if and only if $\mathrsfs{N}(f)$ is continuous. 
% free	in the category of complete modules over a cis $\cis$ with continuous morphisms.  
\end{corollary}

%\begin{proof}
%Let $M$ be a completable $\cis$-module and $f : M \rightarrow N$ be a continuous morphism from $M$ to some complete $\cis$-module $N$. 
%We need to show the existence of a unique continuous morphism $\mathrsfs{N}(f) : \mathrsfs{N}(M) \rightarrow N$ such that $f = \mathrsfs{N}(f) \circ i$. 
%\end{proof}

%\begin{example}\label{ex:dmc}
%Let $M$ be a module over $\cis$. Then there exists a complete module $\overline{M}$ over $\cis$ and an injective linear map $i : M \rightarrow \overline{M}$ such that
%\begin{itemize}
%	\item $\overline{M}/i(M)$ is a module extension over $\cis$, 
%	\item $i(\bigoplus X) = \bigoplus i(X)$, for every subset $X$ of $M$ with a supremum, 
%	\item if $M$ is a complete module, then $i$ is an isomorphism of complete modules.  
%\end{itemize}
%\end{example}

%\textcolor{myred}{C'EST PLUS COMPLIQUE QUE CA ! Pour que $M$ puisse s'injecter dans un complete module, il faut déjà nécessairement que $M$ vérifie l'équation $x.\bigoplus T = \bigoplus (x.T)$ me semble-t-il... } 

We call a pair $\overline{M}/M$ an \textit{extension} over $\cis$ (we shall also speak of the extension $\overline{M}$ of $M$ over $\cis$) if $\overline{M}$ is a complete module over $\cis$ and $M$ is a submodule of $\overline{M}$. 
In this situation, $M$ is necessarily completable. 
The extension is \textit{short} if, for all $y \in \overline{M}$, there is some $x \in M$ such that $y \leqslant x$. 
This condition restricts the ``size'' of $\overline{M}$ and will reveal its importance in the next section. 
Also, the extension is \textit{cut-stable} if the map $i : M \ni x \mapsto x \in \overline{M}$ is cut-stable; in this case, if a subset (resp.\ an $\F$-set) of $M$ has a supremum (resp.\ an infimum) in $M$, then it coincides with its supremum (resp.\ its infimum) in $\overline{M}$. 
For a completable module $M$, the normal completion leads to a short and cut-stable extension $\mathrsfs{N}(M)/i_M(M)$ (see Remark~\ref{rk:ub}). 

\begin{example}\label{ex:ccplus}
Let $E$ be a Hausdorff topological space, let $C_c^+$ be the set of compactly-supported continuous maps from $E$ to $\scif$, and let $L^+$ be the set of lower-semi\-continuous maps from $E$ to $\sci$. Then $L^+/C_c^+$ is an extension over $\scit$ that is neither short nor cut-stable in general. Figure~\ref{fig:Schema1Chapitre3} gives a sequence of continuous functions on $E = [0,1]$, whose supremum is $x \mapsto 1$ in $C_c^+$, and is $x \mapsto 1_{(0,1)}(x)$ in $L^+$. 
\begin{figure}
	\centering
		\includegraphics[width=0.90\textwidth]{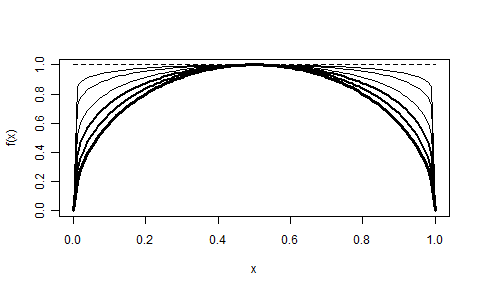}
	\caption{A nondecreasing sequence of continuous functions on $[0, 1]$. }
	\label{fig:Schema1Chapitre3}
\end{figure}
\end{example}

\section{Residuated forms on a module extension}\label{sec:rmf}

This section is expressed in the language $\PFilters$ of principal filters, %$\F$ is again a union-complete filter selection, 
and $\overline{M}/M$ is an extension over an %\textcolor{myblue}{STABLY-CONTINUOUS}
 idempotent semifield $\cis$. Henceforth, all suprema of subsets of $M$ or $\overline{M}$ are taken in $\overline{M}$. 
%We say that the extension $\overline{M}/M$ is \textit{archimedean} if, for all $x \in M$, $c \in \overline{M}$ with $c \neq 0$, there exists some $t \in \cis$ such that $x \leqslant c.t$. 
%An \textit{ideal} of $\overline{M}/M$ is an ideal of $M$. 
%A \textit{principal ideal} of $\overline{M}/M$ is an ideal of $M$ such that $I = \downarrow\!\! a \cap M$, for some $a \in \overline{M}$. 
%\subsection{Residuated forms and their representation}
A map $v : M \rightarrow \cis$ is \textit{residuated on $\overline{M}/M$} if there exists a map $w : \cis \rightarrow \overline{M}$ %called the \textit{adjoint} of $v$ and 
satisfying 
\begin{equation}\label{eq:res1}
x \leqslant w(t) \Longleftrightarrow t \geqslant v(x), 
\end{equation}
for all $x\in M$, $t \in \cis$. 
In this case, there exists a least map $w$ such that Equivalence~(\ref{eq:res1}) holds, called the \textit{adjoint} of $v$ with respect to $\overline{M}/M$, denoted by $v^{\#}$, and defined by 
$$
v^{\#}(t) = \bigoplus \{ x \in M : t \geqslant v(x) \}, 
$$
for all $t \in \cis$, where the supremum is taken in $\overline{M}$. %Residuated maps as defined here are related to Galois connections, see Ern\'e et al.\ \cite{Erne93}. 

\begin{lemma}
A map that is residuated on a short extension of $M$ is residuated on each extension of $M$. 
%The notion of residuated map on $M$ does not depend on the choice of the extension $\overline{M}$ of $M$. % of $M$. 
%Hence a map is a residuated form on $\overline{M}/M$ if and only if it is a residuated form on $\mathrsfs{N}(M)/M$. 
\end{lemma}

\begin{proof}
Let $\overline{M}/M$ be a short extension. 
%We use the notations of Corollary~\ref{coro:injec}, recalled by 
Consider the following commutative diagram: 
\begin{diagram}
M            & \rTo_{i}            & \overline{M}  \\
\dTo<{i_M}     & \ruDashto_{\overline{i}} &    \\
\mathrsfs{N}(M) &                     &    \\
\end{diagram}
where $\overline{i}$ is defined by $\overline{i}(X) = \bigoplus X$. % :  := \mathrsfs{N}(i)$. %In this proof we identify $\mathrsfs{N}(M)$ with $$
Let $v : M \rightarrow \cis$ be a residuated map on $\overline{M}/M$. 
Then $v$ admits an adjoint $v^{\#} : \cis \rightarrow \overline{M}$. 
We first show that $v$ is residuated on $\mathrsfs{N}(M)/M$. 
%We show that $v$ has an adjoint with respect to $\overline{M}/M$ if and only if it has an adjoint with respect to the  extension $\mathrsfs{N}(M)/M$. 
%Suppose first that $v$ admits an adjoint $v^{\#} : \cis \rightarrow \overline{M}$. 
If $t \in \cis$, the subset $I_t = \{ x \in M : t \geqslant v(x) \}$ is upper-bounded (by $v^{\#}(t)$) in $\overline{M}$, hence also in $M$ since $\overline{M}/M$ is short. %by definition of an extension. 
Thus, $I_t$ admits a supremum in $\mathrsfs{N}(M)$, that we denote by $w(t)$. Since $\overline{i}$ preserves arbitrary existing suprema, $\overline{i}(w(t)) = \bigoplus \{ x \in M : t \geqslant v(x) \}$, where the supremum is taken in $\overline{M}$. Thus, $\overline{i} \circ w = v^{\#}$. 
We show that Equivalence~(\ref{eq:res1}) holds. Clearly, $t \geqslant v(x)$ implies $x \leqslant w(t)$. Conversely, assume that $x \leqslant w(t)$. Composing by $\overline{i}$, we get $x \leqslant v^{\#}(t)$, so that $t \geqslant v(x)$. This proves that $v$ is residuated on $\mathrsfs{N}(M)/M$. 
%$w$ is the adjoint of $v$ with respect to $\mathrsfs{N}(M)/M$. 

Now let $\tilde{M}/M$ be some extension of $M$, and consider the related commutative diagram: 
\begin{diagram}
M            & \rTo_{j}            & \tilde{M}  \\
\dTo<{i_M}     & \ruDashto_{\tilde{j}} &    \\
\mathrsfs{N}(M) &                     &    \\
\end{diagram}
where $j : M \ni x \mapsto x \in \tilde{M}$ and $\tilde{j} : X \mapsto \bigoplus X$. % := \mathrsfs{N}(j)$. 
%If $w$ still denotes the adjoint of $v$ with respect to $\mathrsfs{N}(M)/M$, 
Then one can show that 
$$
x \leqslant \tilde{j} \circ w(t) \Leftrightarrow t \geqslant v(x), 
$$ 
for all $x \in M, t \in \cis$, so $v$ is residuated on $\tilde{M}/M$.  
\end{proof}

We define a \textit{residuated form} on $\overline{M}/M$ as a homogeneous residuated map on $\overline{M}/M$. 
A map $v : M \rightarrow \cis$ is \textit{non-degenerate} on $\overline{M}/M$ if $\{ x \in M  : 1 \geqslant v(x) \}$ has an upper-bound in $\overline{M}$. 

%\begin{lemma}\label{lem:rescont}
%Let $M$ be a complete module over a cis $\cis$. Then a map $v : M \rightarrow \cis$ is a smooth residuated form if and only if it is a non-degenerate continuous linear form. 
%\end{lemma}

%\begin{proof}
%Necessity is clear. For sufficiency, assume that $v$ is a non-degenerate continuous linear form, and let $w : \cis \rightarrow M$ such that $w(t) = \bigoplus \{ x \in M : t \geqslant v(x) \}$ for all $t \in \cis$. This map is well-defined since $v$ is non-degenerate and homogeneous. 
%If $t \geqslant v(x)$, then $x \leqslant w(t)$ by definition of $w$. Conversely, if $x \leqslant w(t)$, then $v(x) \leqslant v(w(t))$,  and since $v$ preserves arbitrary existing suprema, $v(x) \leqslant \bigoplus \{ v(x') : x' \in M, t \geqslant v(x') \} \leqslant t$. This proves that $w$ is the adjoint of $v$ with respect to $M/M$, hence $v$ is a (smooth) residuated form. 
%\end{proof}

%The next lemma identifies residuated forms with \textit{b-linear functionals} in the sense of Litvinov et al.\ \cite{Litvinov01}. 

%\textcolor{myred}{PRECISER les questions de smoothness dans la prop suivante}

\begin{lemma}
Let $\overline{M}/M$ be a short extension over a cis $\cis$. 
We suppose that every element of $\overline{M}$ can be expressed as the supremum in $\overline{M}$ of elements of $M$. 
Then a map $v : M \rightarrow \cis$ is a residuated form on $\overline{M}/M$ if and only if $v$ is non-degenerate on $\overline{M}/M$ and extends to a residuated form on $\overline{M}$. %the normal completion of $M$. 
\begin{diagram}
M            & \rTo_{v}            & \cis  \\
\dTo<{i}   & \ruDashto_{\bar{v}} &       \\
\overline{M} &                     &       \\
\end{diagram}
\end{lemma}

\begin{proof}
Let $v$ be a residuated form on $\overline{M}/M$ with adjoint $v^{\#} : \cis \rightarrow \overline{M}$. Then $\{ x \in M : 1 \geqslant v(x) \}$ is upper-bounded (by $v^{\#}(1)$) in $\overline{M}$, so $v$ is non-degenerate on $\overline{M}/M$. Moreover, if $y \in \overline{M}$, then $y$ is upper-bounded by some $x \in M$ since $\overline{M}/M$ is short, so the subset $\{ t \in \cis : y \leqslant v^{\#}(t) \}$ of $\cis$, which contains $t = v(x)$, is nonempty. 
Thus, we can define the map $\overline{v} : \overline{M} \rightarrow \cis$ by $\overline{v}(y) = \bigwedge \{ t \in \cis : y \leqslant v^{\#}(t) \}$. Since  every element of $\overline{M}$ can be expressed as the supremum in $\overline{M}$ of elements of $M$, we have 
$$
y \leqslant v^{\#}(t) \Leftrightarrow t \geqslant \overline{v}(y), 
$$
for all $y \in \overline{M}, t \in \cis$. So we obtain that 
%$\overline{v}$ admits $v^{\#}$ as adjoint with respect to $\overline{M}/\overline{M}$, 
$\overline{v}$ is a residuated form on $\overline{M}$. 
%we get by Lemma~\ref{lem:rescont} that $\overline{v}$ is a continuous linear form on $\overline{M}$ extending $v$. 

Conversely, assume that a map $v : M \rightarrow \cis$ is non-degenerate on $\overline{M}/M$ and extends to a residuated form $\overline{v}$ on $\overline{M}$. %$\overline{v}$ is non-degenerate, hence a smooth residuated form on $\overline{M}$ by Lemma~\ref{lem:rescont}. 
If $w$ is the adjoint of $\overline{v}$ (with respect to $\overline{M}/\overline{M}$), then 
$$
x \leqslant w(t) \Leftrightarrow t \geqslant v(x), 
$$
for all $x \in M, t \in \cis$, 
%$w$ is also the adjoint of $v$ with respect to $\overline{M}/M$, 
so $v$ is a residuated form on $\overline{M}/M$. 
% $v$ is homogeneous; this implies that $\{ x \in M : t \geqslant v(x) \}$ is upper-bounded for all $t \in \cis$, so we can define $v^{\#}(t)$ as the supremum in $\overline{M}$ of the subset $\{ x \in M : t \geqslant v(x) \}$. By continuity of $\overline{v}$, we have $\overline{v}(v^{\#}(t)) = \bigoplus \overline{v}(\{ x \in M : t \geqslant v(x) \}) = \bigoplus v(\{ x \in M : t \geqslant v(x) \}) \leqslant t$. Thus, Equivalence~(\ref{eq:res}) is satisfied, meaning that $v$ is a residuated form on $M$. 
\end{proof}

%\textcolor{myblue}{Est-ce que ma définition de residuated form dépend du choix de $\overline{M}$ ?}

%Assume that $\cis$ is a cis and consider the complete dioid $\overline{\cis} = \cis \cup \{\top\}$ extending $\cis$ (see the definitions in Example~\ref{ex:barl}). 
%, in the sense that every nonempty subset of $\cis$ has an infimum. If $\cis \neq \{0,1\}$, then $\cis$ has no top element, in which case we add some element $\top$ to $\cis$ and define $\overline{\cis} = \cis \cup \{\top\}$. Naturally extending $\oplus$ and $\times$ to $\overline{\cis}$ by $x \oplus \top = \top \oplus x = \top$ and $x.\top = \top .x = \top$ if $x \neq 0$, $0.\top = \top.0 = 0$ otherwise, we see that $\overline{\cis}$ has the structure of a complete dioid. 
%We also define $\top^{-1} = 0$ and $0^{-1} = \top$. 
%We also assume that $\cis$ is \textit{up-complete}, in the sense that every nonempty subset of $\cis$ has a supremum. 
%Let $\overline{M}/M$ is a extension over a cis $\cis$. 
With respect to the filter selection $\PFilters$, we say that an element $c \in \overline{M}$ is \textit{archimedean in $\overline{M}/M$} if
\begin{itemize}
	\item the subset $F_c(x)$ is nonempty for all $x \in M$, %there exists some $t \in \cis$ such that $c.t \geqslant x$.
	\item if the infimum of $F_c(x)$ is zero, then $F_c(x) \supset \cis\setminus\{0\}$,   
\end{itemize}
where $F_c(x)$ denotes the subset $\{ t \in \cis : c.t \geqslant x \}$. 
The next lemma can be proved along the same lines as Lemma~\ref{lem:arch1}. 
%\begin{itemize}
%	\item for all $x \in M$, there exists some $t \in \cis$ such that $c.t \geqslant x$;
%	\item for all $t \in \cis$, $t \gg 1$ implies $c.t \gg c$. 
%\end{itemize}

%An element $c \in \overline{M}$ is \textit{archimedean in $\overline{M}/M$} if, for all $x \in M$, there exists some $t \in \cis$ such that $x \leqslant c.t$. 
%Also, $c \in M$ is \textit{archimedean in $M$} if $c$ is archimedean in $M/M$. 

%\begin{example}
%Let $\overline{M}/M$ be an extension over a cis $\cis$, and let $v : M \rightarrow \cis$ be a linear form. Then the supremum $c$ (in $\overline{M}$) of the set $X = \{ x \in M : 1 \geqslant v(x) \}$ is archimedean in $\overline{M}/M$. To see this, let $x \in M$ and let $t = v(x)$. If $t = 0$, then $x \in X$, so $x \leqslant c.1$ by definition of $c$. If $t \neq 0$, then $v(x.t^{-1}) = 1$, so $x.t^{-1} \in X$. This implies $x.t^{-1} \leqslant c$ (still by definition of $c$), so that $x \leqslant c.t$. 
%\end{example}

%\textcolor{myred}{ATTENTION dans le résultat précédent, dans la déf de archimedien, et dans le lemme suivant, les $F$-sets sont dans $M$ et non pas dans $\overline{M}$... Donc $v$ ne s'étend pas vraiment à une fonction continue sur $\overline{M}$, car la smoothness ne s'étend pas. }

\begin{lemma}
Let $\overline{M}/M$ be an extension over a cis $\cis$, 
and let $v : M \rightarrow \cis$ be a residuated form on $\overline{M}/M$. Then the supremum of $\{ x \in M : 1 \geqslant v(x) \}$ is archimedean in $\overline{M}/M$. %(where $M$ is considered as an extension over itself). 
\end{lemma}

The innovation of this section mainly relies on highlighting the role of the following concept in the representation of continuous linear forms. An extension $\overline{M}/M$ is \textit{meet-continuous} if 
$$
x \wedge \bigoplus I = \bigoplus \downarrow \!\!x \cap I, 
$$
for all $x\in M$ and all ideals $I$ of $M$ with an upper-bound in $\overline{M}$. 

\begin{example}[Example~\ref{ex:ccplus} continued]\label{ex:ccplus2}
Let $E$ be a topological space. We still denote by $L^+$ the set of lower-semi\-continuous maps from $E$ to $\sci$. If $M$ is some submodule of the set of continuous maps from $E$ to $\scif$, then the extension $L^+/M$ is meet-continuous. 
Before proving this assertion, note that the supremum in $L^+$ coincides with the pointwise supremum. Now let $f \in M$ and let $I$ be an ideal in $M$. We want to show that $f \wedge \bigoplus I \leqslant \bigoplus \downarrow \!\!f \cap I$, i.e.\ that $f(x) \wedge \bigoplus_{g \in I} g(x) \leqslant \bigoplus_{h \in I, h \leqslant f} h(x)$. For this purpose, let $s < f(x) \wedge \bigoplus_{g \in I} g(x)$. There is some $g \in I$ such that $s < f(x) \wedge g(x)$. Then the map $h = f \wedge g$ is continuous, is in $I$ and satisfies $h \leqslant f$ and $s < h(x)$, so the claim follows. 
\end{example}

A map $f : M \rightarrow \cis$ is \textit{continuous} on $\overline{M}/M$ if, for every subset $X \subset M$ such that $\bigoplus X \in M$, $f(X)$ has a supremum in $\cis$ satisfying $f(\bigoplus X) = \bigoplus f(X)$. 
If $x \in M$ and $c \in \overline{M}$, we write $\langle c, x \rangle$ 
%$x\backslash c$ 
for the infimum of $\{ t \in \cis : c.t \geqslant x \}$ whenever this set is nonempty, and $\langle c, x \rangle = \top$ otherwise. 
The following theorem shows that, under conditions different than those of Theorem~\ref{thm:litv}, the representation $v(\cdot) = \langle c, \cdot \rangle$ still holds, at the price that $c$ no longer needs to belong to $M$. 
But it is actually important to authorize $c$ to be outside $M$, in order to encompass the (idempotent) Riesz representation theorem (see Theorem~\ref{thm:riesz1} below). %; this result states that continuous linear forms defined on the $\mathbb{R}_+$-module $C_c^+$ of nonnegative compactly-supported continuous maps on some locally-compact Hausdorff topological space admit a density \textit{outside} $C_c^+$ in general). 
%, which as a map is upper-semicontinuous (but not continuous in general). 

\begin{theorem}\label{supergene}
Suppose that $\overline{M}/M$ is a meet-continuous extension over a cis $\cis$, and let $v : M \rightarrow \cis$ be a linear form on $M$. Then the following conditions are equivalent: 
\begin{enumerate}
	\item\label{supergene0} $v$ is residuated on $\overline{M}/M$, 
	\item\label{supergene1} $v$ is non-degenerate continuous on $\overline{M}/M$, 
%	\item\label{supergene1bis} $v$ extends to a non-degenerate continuous linear form on the complete submodule $\downarrow\!\! M$ of $\overline{M}$, %est une forme linéaire continue dans $Y$ 
	\item\label{supergene2} $v(\cdot) = \langle c, \cdot \rangle$, for some archimedean element $c$ in $\overline{M}/M$. %(\cdot \backslash c)^{-1}$. 
\end{enumerate}
If any of these conditions is satisfied, then the supremum of $\{ 1 \geqslant v \}$ in $\overline{M}$ is the least $c$ satisfying $v(\cdot) = \langle c, \cdot \rangle$. % (\cdot \backslash c)^{-1}$. 
\end{theorem}

\begin{proof}
Assume that $v$ is non-degenerate and continuous. Let $t \in \cis$ and $I_t = \{ x \in M : t \geqslant v(x) \}$. Clearly $I_t$ is an ideal, let $w(t)$ denote its supremum in $\overline{M}$. 
To prove that $v$ is residuated on $\overline{M}/M$, we have to show that, if $x \in M$ and $x \leqslant w(t)$, then $t \geqslant v(x)$, i.e.\ $x \in I_t$. 
%If $x \in I$ then $x \in \downarrow\!\!w(t) \cap M$. Conversely let $x \in \downarrow\!\! w(t) \cap M$. Then, 
If $\overline{M}/M$ is meet-continuous, then $x = x \wedge w(t) = x \wedge \bigoplus I_t = \bigoplus \downarrow \!\!x \cap I_t \in M$, and using the fact that $v$ is continuous we get $v(x) = \bigoplus_{y \in \downarrow x \cap I_t} v(y) \leqslant t$, i.e.\ $x \in I_t$. This proves that (\ref{supergene1}) implies (\ref{supergene0}), and the converse implication is clear. 
The equivalence between (\ref{supergene0}) and (\ref{supergene2}) can be proved along the same lines as in the proof of Theorem~\ref{thm:litv}. 
\end{proof}

\begin{remark}\label{rk:unique}
If every element of $\overline{M}$ can be expressed as the supremum in $\overline{M}$ of elements of $M$, there is a \textit{unique} such element $c$ in the previous theorem. Indeed, assume that, for some archimedean elements $b, c$ in $\overline{M}/M$, we have $\langle b, x\rangle = \langle c, x \rangle$ for all $x \in M$. Then $x \leqslant b \Leftrightarrow 1 \geqslant \langle b, x \rangle \Leftrightarrow 1 \geqslant \langle c, x \rangle \Leftrightarrow x \leqslant c$, for all $x \in M$, so that $\downarrow\!\! b \cap M = \downarrow\!\! c \cap M$. This implies that $b = \bigoplus \downarrow\!\! b \cap M = \bigoplus \downarrow\!\! c \cap M = c$. 
\end{remark}

%Theorem~\ref{supergene} indeed generalizes Theorem~\ref{thm:cohen}, for a $\cis$-module $M$ can always be embedded in a complete $\cis$-module, for instance the Dedekind--MacNeille completion of $M$ (see Example~\ref{ex:dmc}). 
As a direct application, we shall reprove an idempotent version of the Riesz representation theorem in Section~\ref{sec:riesz}. 
%give the following idempotent Riesz representation theorem.  

%\textcolor{myblue}{ON PEUT DONNER UN THM DISANT QUE $M$ et $M'$ sont isomorphes. }

\section{The Riesz representation theorem}\label{sec:riesz}

%\subsection{Definition of linear forms}

%Let $E$ be a set and $M$ be a set of functions from $E$ to $\scif$, which contains the $0$ function and is closed under multiplication by nonnegative scalars and the formation of maxima. 
%A  linear form on $M$ is a map $V : M \rightarrow \sci$ such that $V(0) = 0$ and 
%\begin{equation}\label{additivite1}
%V(t.f \oplus g) = t.V(f) \oplus V(g),
%\end{equation}
%for all $f, g \in M$ and $t \in \scif$.  

%\subsection{Maxitive forms on topological spaces}

In this section, we aim at proving Riesz representation theorems for the Shilkret integral with the help of Theorem~\ref{supergene}. 
The filter selection implicitly used here is $\PFilters$, i.e.\ the one that selects principal filters. 
	
Let $E$ be a Hausdorff topological space, $\mathrsfs{G}$ (resp.\ $\mathrsfs{B}$) be the collection of open subsets (resp.\ Borel subsets) of $E$, and $C_c^+$ be the set of nonnegative compactly-supported continuous maps from $E$ to $\scif$. 
The next theorem is of historical importance, for (part of) it was originally stated by Choquet \cite[Paragraph~53.1]{Choquet54} without proof; the first proof is due to Kolokoltsov and Maslov \cite[Theorem~1]{Kolokoltsov87}. 
The reader can also refer to Puhalskii \cite[Theorem~1.7.21]{Puhalskii01}. 

\begin{lemma}[Urysohn]
Let $E$ be a locally-compact Hausdorff space. 
If $K \subset U \subset E$ with $K$ compact and $U$ open, then there exists a compactly-supported continuous map $f : E \rightarrow [0, 1]$ such that $f(x) = 1$ for all $x \in K$ and $\overline{\{ x \in E : f(x) > 0 \}} \subset U$. 
\end{lemma}

\begin{proof}
This is customarily proved by using the fact that the one-point compactification of $E$ is a normal space, see e.g.\ Aliprantis and Border \cite[Corollary~2.74]{Aliprantis06}. 
\end{proof}

\begin{lemma}\label{lem:lscsup}
Let $E$ be a locally-compact Hausdorff space. 
Every lower-semi\-continuous map $g : E \rightarrow \sci$ is a supremum of elements of $C_c^+$. 
\end{lemma}

\begin{proof}
Let $s \in \scif$ and $x \in E$ be such that $s < g(x)$. Since $\{ g > s \}$ is an open subset, there is some $f_1 \in C_c^+$ such that $s < f_1(x) < g(x)$ and $f_1 = 0$ on $\{ g \leqslant s \}$, by Urysohn's lemma. Now $\{ g > f_1 \}$ is also open, so there is some $f_2 \in C_c^+$ such that $f_2(x) = f_1(x)$ and $f_2 = 0$ on $\{ g \leqslant f_1 \}$. This proves that the map $f_{s, x} = f_1 \wedge f_2$ is in $C_c^+$ and satisfies $f_{s, x} \leqslant g$ and $s < f_{s, x}(x)$. As a consequence, one can see that $g$ is the pointwise supremum of all such maps $f_{s, x}$. %, where $s, x$ satisfy $s < g(x)$. 
\end{proof}

%\begin{lemma}
%Let $E$ be a Tychonoff space. Every lower-semicontinuous map $f : E \rightarrow \sci$ is a supremum of $\scif$-valued continuous maps. 
%\end{lemma}

%\begin{proof}
%Let $h : E \rightarrow \sci$ be the lower-semicontinuous map defined as the supremum of all $\scif$-valued continuous maps below $f$. Then $h \leqslant f$. Suppose that $h(x) < f(x)$ for some $x \in E$, and let $s$ such that $h(x) < s < f(x)$. Then $x$ is in the open subset $\{ f > s \}$. Since $E$ is Tychonoff, there is some continuous map $g : E \rightarrow \scif$ such that $g(x) = 1$ and $g = 0$ on $\{ f \leqslant s \}$. 
%\end{proof}

%We denote by $C_c^+$ be the set of compactly supported continuous maps from $E$ to $\scif$. 

%\begin{theorem}\cite[Theorem~1.7.21]{Puhalskii01}
%Assume that $E$ is a locally-compact Hausdorff space, and let $V : M \rightarrow \sci$ be a linear form on $C_c^+$. 
%Then there exists some upper-semicontinuous map $c  : E \rightarrow \sci$ such that 
%$$
%V(f) = \max_{x\in E} c(x) . f(x), 
%$$
%for all $f \in M$, where the $\max$ operator points out that the supremum is reached. 
%\end{theorem}

\begin{theorem}[Improves \protect{\cite[Theorem~1.7.21]{Puhalskii01}}]\label{thm:riesz1}
Let $E$ be a locally-com\-pact Hausdorff space, and let $V : C_c^+ \rightarrow \scif$ be a linear form on $C_c^+$. 
Then there exists a unique regular maxitive measure $\nu$ on $\mathrsfs{B}$ %$2^E$ 
such that %$V$ is exactly the Shilkret integral 
$$
V(f) = \ddint f\,d\nu,
$$
for all $f \in C_c^+$. 
Moreover, $\nu$ takes finite values on compact subsets of $E$. 
%Moreover, $\nu$ is completely maxitive, and there is some usc map $c  : E \rightarrow \scif$ such that 
%$$
%V(f) = \max_{x\in E} c(x) . f(x), 
%$$
%for all $f \in C_c^+$, where the $\max$ operator points out that the supremum is reached. 
\end{theorem}

%\textcolor{myblue}{JE CROIS QU'ON A BESOIN que V soit à valeurs finies, notamment pour pouvoir dire que $\varepsilon V(h)$ tend vers zéro.}

\begin{proof}
%Using Dini's theorem, it is not difficult to show that $V$ is a continuous linear form on $M$ (see e.g.\ Puhalskii \cite[Theorem~1.7.21]{Puhalskii01}). If $\overline{M}$ denotes the set of compactly supported lower-semicontinuous maps $g : E \rightarrow \sci$, one can see that $\overline{M}/M$ is an archimedean  extension. 
The functional $V$ is a linear form on the $\scit$-module $C_c^+$. If $\overline{M} = L^+$ is the module of %compactly supported 
$\sci$-valued lower-semi\-continuous maps and $M = C_c^+$, then $\overline{M}/M$ is a meet-continuous extension by Example~\ref{ex:ccplus2}, so 
Theorem~\ref{supergene} applies %\textcolor{myred}{BOF car l'extension n'est pas archimédienne...} 
if we show that $V$ is continuous and non-degenerate on $\overline{M}/M$. 
Non-degeneracy of $V$ on $\overline{M}/M$ is ensured by the existence of arbitrary suprema in $\overline{M}$. %for if $V(f) \leqslant 1$, then $f \leqslant c^{-1}$, where $c$ is the upper-semicontinuous map defined by 
%$$
%c(x) = \bigwedge_{f \in M : f(x) > 0} \frac{V(f)}{f(x)}, 
%$$
%for all $x \in E$. 
For the continuity of $V$, let $(f_j)_{j \in J}$ be a nondecreasing net of elements of $C_c^+$ such that $f := \bigoplus_{j \in J} f_j \in M$, where the supremum is taken in $\overline{M}$. In particular, $(f_j)_{j \in J}$ converges pointwise to $f$, see Example~\ref{ex:ccplus2}. 
We want to prove that 
$
V(f) = \bigoplus_{j\in J} V(f_j).
$
So let $1> \varepsilon > 0$, let $K_{\varepsilon}$ be the compact set $\{ f \geqslant \varepsilon \}$, and define $h_j$ on $K_{\varepsilon}$ by $h_j(x) = f_j(x)/f(x)$. Then $h_j \in C_c^+(K_{\varepsilon})$ and $(h_j)_{j \in J}$ is a nondecreasing net converging to $1$ pointwise.  Applying Dini's Theorem, the convergence is uniform on $K_{\varepsilon}$, hence there is some $j_0 \in J$ such that $1 \leqslant \varepsilon + h_{j_0}$ on $K_{\varepsilon}$. Thus, $f \leqslant \varepsilon.f + f_{j_0}$ on $K_{\varepsilon}$. 
Let $K$ be a compact set containing $\{ 0 < f < 1 \}$. By Urysohn's lemma, we may find a compactly-supported continuous map $h$ such that $h = 1$ on $K$. 
Then $f \leqslant (\frac{1}{1-\varepsilon} f_{j_0}) \oplus (\varepsilon h)$ on $E$. This implies $V(f) \leqslant (\frac{1}{1-\varepsilon} \bigoplus_{j\in J} V(f_j)) \oplus (\varepsilon . V(h))$, for all $1> \varepsilon > 0$, so $V$ is continuous. By Theorem~\ref{supergene}, there exists some archimedean element $c$ in $\overline{M}/M$ such that $V(f) = \langle c, f \rangle$, %(f\backslash c)^{-1}$, 
for all $f \in M$. Defining the usc map $c^{+}$ by $c^{+}(x) = 1/c(x)$, we have $V(f) = \bigoplus_{x\in E} f(x) . c^{+}(x)$, for all $f \in M$. 
The maxitive measure $\nu$ defined on $\mathrsfs{B}$ by $\nu(B) = \bigoplus_{x \in B} c^{+}(x)$ for all $B \in \mathrsfs{B}$ is regular by \cite[Theorem~3.22]{Poncet12b}, and we have 
$$
V(f) = \ddint f\,d\nu, 
$$
for all $f \in M$. 
If $K$ is a compact subset of $E$, Urysohn's lemma %for locally-compact Hausdorff spaces (see e.g.\ Aliprantis and Border \cite[Corollary~2.74]{Aliprantis06})
 provides a map $f \in M$ such that $f(x) = 1$ for all $x \in K$, so that $c^{+}(x) \leqslant V(f)$ for all $x \in K$. This ensures that $\nu(K) = \bigoplus_{x \in K} c^{+}(x)$ is finite. 
%One can see that archimedeanity of $g$ implies that $g(x) \neq 0$ for all $x$, so that $c$ is finite-valued. 

Uniqueness of $\nu$ is a direct consequence of the uniqueness of $c$, which itself derives from Lemma~\ref{lem:lscsup} and Remark~\ref{rk:unique}. %, we show that we necessarily have 
%$$
%c(x) = \bigwedge_{f \in M : f(x) > 0} \frac{V(f)}{f(x)}, 
%$$
%for all $x \in E$. Let $d$ denote the map defined by the right hand side. Then $c \leqslant d$, and we want to show the converse inequality. 
%Let $x \in E$, and let $G$ be an open subset of $E$ containing $x$. 
%Applying again Urysohn's lemma, there is some map $f \in M$ such that $0 \leqslant f \leqslant 1$, $f(x) = 1$, and  $\{ f > 0 \} \subset G$. The definition of $d$ gives $d(x) \leqslant V(f)$, and since $V(f) = \bigoplus_{t \in [0, 1]} t.\nu(f > t)$, we deduce that $V(f) \leqslant \nu(G)$. Thus, $d(x) \leqslant \bigwedge_{G \ni x} \nu(G)$. Since $\nu$ is regular, we conclude that $d(x) \leqslant c(x)$ for all $x \in E$. 
%As a consequence, $\nu$ satisfies the identity 
%$$
%\nu(B) = \bigoplus_{x \in B} \bigwedge_{f \in M : f(x) > 0} \frac{V(f)}{f(x)}, 
%$$
%hence is uniquely defined. 
\end{proof}

%\begin{remark}
%Given that $V$ admits a density, it is not difficult to show that its maximal density is 
%$$
%c(x) = \bigwedge_{f \in C_c^+ : f(x) > 0} \frac{V(f)}{f(x)}. 
%$$
%Moreover, this formula implies that $c$ is usc. %Also, the use of Urysohn's lemma for locally-compact Hausdorff spaces (see e.g.\ Aliprantis and Border \cite[Corollary~2.74]{Aliprantis06}) guarantees that $c$ is finite-valued. 
%\end{remark}

%\textcolor{myblue}{Preuve de l'unicité ? }

%\begin{lemma}[Urysohn]
%Let $E$ be a locally-compact Hausdorff space. 
%If $K \subset U \subset E$ with $K$ compact and $U$ open, then there exists a compactly-supported continuous map $f : E \rightarrow [0, 1]$ such that $f(x) = 1$ for all $x \in K$ and $\overline{\{ x \in E : f(x) > 0 \}} \subset U$. 
%\end{lemma}

%\begin{proof}
%This is customary proved by using the fact that the one-point compactification of $E$ is a normal space, see e.g.\ Aliprantis and Border \cite[Corollary~2.74]{Aliprantis06}. 
%\end{proof}

%\textcolor{myblue}{Ce lemme sert à dire qu'en fait $c$ est à valeurs dans $\scif$ plutôt que dans $\sci$, i.e.\ $1/c$ ne s'annule et est donc un élément archimédien !! }

% SOURCE : PLANET MATH

In the same line, one can formulate Riesz like theorems for a functional $V$ defined on the set $C_b^+$ of nonnegative bounded continuous maps instead of $C_c^+$. %Puhalskii \cite[Theorem~1.7.25]{Puhalskii01} proved next theorem, but as for Gulinsky \cite{Gulinsky03} the first occurrence of it dates back to
Breyer and Gulinsky \cite{Breyer96} proved the next theorem\footnote{We were not in a position to access this article. }, see also Puhalskii \cite[Theorem~1.7.25]{Puhalskii01} and Gulinsky \cite[Theorem~3.4]{Gulinsky03}. 
%is due to Akian \cite[Theorem~4.8]{Akian99}, except for the last part of it, for which a proof is given. 
%\textcolor{myblue}{Intuitively, if we could define $\nu(G) = V(1_G)$, with $G$ an open subset of $E$, then $\nu$ would be a \textit{tight} $\sigma$-maxitive measure (see Definition~\ref{tension}) with cardinal density $c$, and Equation~(\ref{varadhan}) would become $\nu(f) = \bigoplus_{x \in E} f(x). c(x)$. However, this definition of  $\nu$ is not correct since $1_G$ does not belong to $C_b^+$ in general.} 
%The next theorem untangles this problem. 
A functional $V : C_b^+ \rightarrow \scif$ is \textit{tight} if, for all $\varepsilon > 0$, there is some compact subset $K$ of $E$ such that $V(f) \leqslant \varepsilon \left\|f\right\|$, for each $f \in C_b^+$ that equals $0$ on $K$. % be a $\sigma$-maxitive form on $C_b^+$. 

\begin{theorem}\cite{Breyer96}\label{thm:riesz3}
Assume that $E$ is a Tychonoff space, and let $V : C_b^+ \rightarrow \scif$ be a tight linear form on $C_b^+$ that preserves countable pointwise suprema. 
Then there exists a unique (finite) tight regular maxitive measure $\nu$ on $\mathrsfs{B}$ such that 
$$
V(f) = \ddint f\,d\nu,
$$
for all $f \in C_b^+$. %Moreover, $\nu$ is finite. 
\end{theorem}

\begin{proof}
See e.g.\ Puhalskii \cite[Theorem~1.7.25]{Puhalskii01}. The idea of the proof is to use the Stone--\v{C}ech compactification of $E$ and to apply Theorem~\ref{thm:riesz1}. 
\end{proof}

%\begin{lemma}\label{lem:lscsup2}
%Let $E$ be a Tychonoff space. 
%Every lower-semicontinuous map $g : E \rightarrow \sci$ is a supremum of elements of $C_b^+$. 
%\end{lemma}

%\begin{proof}
%Entirely similar to that of Lemma~\ref{lem:lscsup}. 
%\end{proof}
In order to treat the case of a non-tight linear form on $C_b^+$, we shall assume that the Tychnoff space is also second-countable, i.e.\ is a separable metrizable space. %First we state an analogue of Lemma~\ref{lem:lscsup}. 
A functional $V : C_b^+ \rightarrow \scif$ is \textit{optimal} if, for all nonincreasing sequences $(f_n)_{n \in \mathbb{N}}$ of elements of $C_b^+$ tending pointwise to $0$, the sequence $(V(f_n))_{n \in \mathbb{N}}$ tends to $0$.  

%\textcolor{myred}{ATTENTION A SIGMA-CONTINUOUS CI-DESSOUS, on est toujours pas rapport à des pointwise suprema... }

\begin{theorem}\label{thm:riesz4}
Assume that $E$ is a separable metrizable space, and let $V : C_b^+ \rightarrow \scif$ be a linear form on $C_b^+$ that preserves countable pointwise suprema. 
Then there exists a unique (finite) regular maxitive measure $\nu$ on $\mathrsfs{B}$ such that 
$$
V(f) = \ddint f\,d\nu,
$$
for all $f \in C_b^+$. %Moreover, $\nu$ is finite. 
Moreover, if $E$ is Polish, %(hence metric, normal and Tychonoff), %thanks to Theorem~\ref{caspart} 
then the following conditions are equivalent: %Si de plus $V$ vérifie 
\begin{itemize}
	\item $V$ is optimal, %$\lim_n V(f_n) = 0$, for all $f_1, f_2, \ldots \in C_b^+$ such that $f_n \downarrow 0$, 
	\item $V$ is tight, 
	\item $\nu$ is tight, 
	\item $c^{+}$ is upper-compact, %\{ c^{+} \geqslant t \}$ is compact, for all $t>0$, 
\end{itemize}
where $c^{+}$ is the maximal cardinal density of $\nu$. 
\end{theorem}

\begin{proof}
We denote $L^+$ by $\overline{M}$ and $C_b^+$ by $M$. 
As in the proof of Theorem~\ref{thm:riesz1}, $\overline{M}/M$ is a meet-continuous extension by Example~\ref{ex:ccplus2}, $V$ is a non-degenerate linear form on $\overline{M}/M$, and we want to prove that 
\begin{equation}\label{eq:supcom}
V(f) = \bigoplus_{j\in J} V(f_j), 
	\end{equation}
for all nondecreasing nets $(f_j)_{j \in J}$ in $M$ such that $f := \bigoplus_{j \in J} f_j \in M$, where the supremum is taken in $\overline{M}$. %In particular, $(f_j)_{j \in J}$ converges pointwise to $f$, see Example~\ref{ex:ccplus2}. 
Let $q$ be a nonnegative rational number. The open subset $\{ f > q \}$ is covered by the family of open subsets $\{ f_j > q \}$, $j \in J$. Since $E$ is separable metrizable, it is second-countable, so we can extract a countable subcover and write $\{ f > q \} = \bigcup_{j \in N_q} \{f_j > q \}$, where $N_q$ is a countable subset of $J$. Defining $N$ as the union of all $N_q$, which is countable, we see that $f = \bigoplus_{j \in N} f_j$. Since $V$ preserves countable pointwise suprema in $M$, Equation~(\ref{eq:supcom}) holds, so $V$ is continuous on $\overline{M}/M$. The existence of $\nu$ now follows from the same argument as in the proof of Theorem~\ref{thm:riesz1}. 
Since $E$ is separable metrizable, $E$ is normal; using Urysohn's lemma for normal spaces (see e.g.\ \cite[Theorem~2.46]{Aliprantis06}), we can show that every $\sci$-valued lower-semi\-continuous map is a supremum of elements of $C_b^+$, as a perfect analogue of Lemma~\ref{lem:lscsup}, and this leads to the uniqueness of $\nu$. 

Now suppose that $E$ is Polish. %Then $\mathrsfs{G} \subset \mathrsfs{F}_{\sigma}$. 
Assume that $\{ c^{+} \geqslant t \}$ is not compact, for some $t>0$. Then there exists some $\varepsilon > 0$ and some sequence $(x_n)$ of elements of  $\{ c^{+} \geqslant t \}$ such that $d(x_m,x_n) > \varepsilon$, for all $m \neq n$. Since $E$ is Polish, one can find some countable family of open balls with radius $\varepsilon /2$ covering $E$. Let $B_k \ni x_k$ be one of these balls containing $x_k$. Let $f_k \in C_b^+$ such that $f_k(x_k) = 1/t$ and $f_k = 0$ on $E \setminus B_k$. Then $g_n := \bigoplus_{k \geqslant n} f_k$ tends pointwise to $0$. But $V(g_n) \geqslant g_n(x_n) . c^{+}(x_n) \geqslant 1$, so $V$ is not optimal. 

Conversely, assume that $c^{+}$ is upper-compact and that $V$ is not optimal. So let $(f_n)$ be a nonincreasing sequence of elements of $C_b^+$ that tends pointwise to zero, and assume that $\bigwedge_{n \in \mathbb{N}} V(f_n) > 0$. Then there exists some $t>0$ such that $V(f_n) > t$, for all $n$. We deduce the existence of some $x_n \in E$ such that $f_n(x_n) . c^{+}(x_n) > t$. 
Since the sequence $(f_n)$ is nonincreasing, $f_n \leqslant f_0$. Also, $f_0$ is upper-bounded by some $u > 0$, so that $x_n$ is in the compact subset $\{ c^{+} \geqslant t/u \}$. This implies that $(x_n)$ clusters to some $x$, % (et on peut supposer quitte à extraire que $x_n \rightarrow x$), 
but this contradicts the fact that $f_m(x_n) \geqslant f_m(x_n) . c^{+}(x_n) \geqslant f_n(x_n) . c^{+}(x_n) > t$ for all $n \geqslant m$. %, et que l'on passe à la limite en $n$ puis en $m$.
%$5) \Leftrightarrow 6)$ a déjà été prouvé, cf. proposition~\ref{tensioneq}. 

By \cite[Theorem~3.25]{Poncet12b}, we know that $\nu$ is tight if and only if $c^{+}$ is upper-compact. If $V$ is tight, then $\nu$ is tight by uniqueness of $\nu$ in Theorem~\ref{thm:riesz3}; the converse statement is obvious. 
\end{proof}

%\textcolor{myred}{REMARQUE SUR STRONGLY LINDELOF}

Bell and Bryc also investigated the case where $E$ is Polish, their result \cite[Theorem~2.1]{Bell01} is encompassed in the previous theorem. 
%This result encompass The case where $E$ is Polish was revisited by Bell and Bryc \cite[Theorem~2.1]{Bell01}. % revisited this question for an optimal linear form $V$ on $C_b^+(E)$, with $E$ a Polish space, %\rightarrow \scif$ satisfying the axioms of Theorem~\ref{thm:??} %with the following axioms: 
%\begin{enumerate}
%	\item normalization: $V(1_E) = 1$, 
%	\item maxitivity: $V(f \oplus g) = V(f) \oplus V(g)$, for all $f, g \in C_b^+$, 
%	\item homogeneity: $V(t. f) = t . V(f)$, for all $f \in C_b^+$, $t \in \scif$, 
%	\item[\textit{4')}] $\lim_n V(f_n) = 0$, for all $f_1, f_2, \ldots \in C_b^+$ such that $f_n \downarrow 0$.
%\end{enumerate}
%They 
%and reproved Equation~(\ref{varadhan}),  % by choosing 
%$$
%c(x) = e^{-I(x)} = \bigwedge_{f\in C_b^+\mbox{, } f(x) > 0} \frac{V(f)}{f(x)}, 
%$$
%$I$ being a \textit{lower-compact} rate function (i.e.\ a lower-semicontinuous map such that the sets $\{I \leqslant t\}$ are compact for all $t > 0$). 

Akian proved a slightly different result for normal spaces %(recall that every normal space is Tychonoff), 
that we merely recall for the sake of completeness. 

\begin{theorem}\cite[Theorem~4.8]{Akian99}\label{thm:riesz2}
Assume that $E$ is a normal space, and let $V : C_b^+ \rightarrow \scif$ be a linear form on $C_b^+$ that preserves countable pointwise suprema. 
Then there exists a unique $\sigma$-maxitive measure $\nu$ on $\mathrsfs{B}$ %$2^E$ 
such that %$V$ is exactly the Shilkret integral 
$$
V(f) = \ddint f\,d\nu,
$$
for all $f \in C_b^+$. 
%Moreover, \textcolor{myblue}{if $E$ is metric, then $\nu$ is regular, and} if $E$ is Polish (hence metric, normal and Tychonoff), %thanks to Theorem~\ref{caspart} 
%then the following conditions are equivalent: 
%\begin{itemize}
%	\item $V$ is optimal, 
%	\item $V$ is tight, 
%	\item $\nu$ is tight, 
%	\item $c^{+}$ is upper-compact, 
%\end{itemize}
%where $c^{+}$ is the maximal cardinal density of the restriction of $\nu$ to $\mathrsfs{G}$.  
\end{theorem}

\begin{remark}[On large deviations]
The idempotent Riesz representation theorem partly originates from large deviation questionings. 
%Varadhan \cite{Varadhan66}, and after him Bell and Bryc \cite{Bell01}, gave a proof of Theorem~\ref{??} %(on the existence of a cardinal density, in the case where $E$ is a topological space) in a certain sense we explain below. Let $E$ be a Polish space (in particular $E$ is second countable, so that Theorem~\ref{??} will apply) endowed with its Borel $\sigma$-algebra. %Denote by $C_b^+$ the set of bounded continuous maps from $E$ to $\scif$. 
Varadhan \cite{Varadhan66} was interested in the functional defined on the set $C_b^+(E)$ of nonnegative bounded continuous maps, for some Polish space $E$,  by 
$$
V(f) = \lim_{n\rightarrow \infty} \left(\int_E f^{1/\alpha_n} \, d\mu_n \right)^{\alpha_n}, 
$$
whenever the limit exists, 
where $(\mu_n)$ is a sequence of probability measures on $E$ satisfying a large deviation principle, and $\alpha_n \downarrow 0$ when $n \uparrow \infty$. He proved the representation 
\begin{equation}\label{varadhan}
V(f) = \bigoplus_{x \in E} (f(x) e^{- I(x)}),
\end{equation}
where $I : E \rightarrow \sci$ is the (lower-semi\-continuous) \textit{rate function}, that governs large deviations. 
For more on the links between large deviation principles and maxitive measures, we refer the reader to Puhalskii \cite{Puhalskii94, Puhalskii01}, B.\ Gerritse \cite{Gerritse96}, Akian \cite{Akian99}, Nedovi\'c et al.\ \cite{Nedovic05}. 
\end{remark}

%The following lemma is stated in \cite[Theorem~1.4.22]{Puhalskii01}, but here we give a different proof:

\section{Conclusion and perspectives}\label{sec:conclusion}

Following Cohen et al.\ \cite{Cohen04}, we could certainly have pushed on the generalization to the use of \textit{reflexive} idempotent semirings\footnote{Or even to the use of idempotent semirings in which every element is the supremum of reflexive elements, so as to generalize both Cohen et al.'s and Litvinov et al.'s approaches. } instead of idempotent semifields, but our main interest here, at least in the first part of the paper, was to stress the role of $\Z$ theory in the gathering of similar but a priori distinct results from different mathematical areas. 

Some results on modules and continuous linear forms are of topological flavour; this aspect will be sharpened in a future work with the examination of topological $\scit$-modules. 

%INTERET QUAND MEME : considerer un module sur un reflexive dioid permet de supposer le dioide complete plutôt que down-complet, donc de s'affranchir de l'hypothese archimedien. 

%GENERALITE MAXI je pense : considérer un idempotent semiring où tout element est le sup d'éléments reflexifs... 

%If $\alpha$ is any cardinal number (and in particular if $\alpha$ is the smallest infinite cardinal $\aleph_0$), all results that are expressed here for \textit{finitely} maxitive maps could be as well transposed to $\alpha$-maxitive maps (with a straightforward definition), provided that ideals are replaced by their appropriate `$\alpha$' counterpart. 

%The concept of maxitive maps introduced in this paper has the advantage to encompass both the notions of maxitive measure and max-linear form. 
%A natural continuation of this paper would be to apply our results to maxitive measures. An upcoming paper \cite{Poncet09c} shall tackle the problem of decomposing maxitive measures into a regular and a singular part. 

\subsection*{Acknowledgments}

I would like to thank Marc Leandri for his advice on a preliminary version of the manuscript. %the precious advice he made on the manuscript.  
I am also indebted in Pr.\ Jimmie D.\ Lawson who made numerous comments and pointed out a mistake in an ealier version of Section~\ref{sec:dmcc}. 

\bibliographystyle{plain}%{apalike}
%\bibliography{C:/LocalTex/BIBLIO}

\def\cprime{$'$} \def\cprime{$'$} \def\cprime{$'$} \def\cprime{$'$}
  \def\ocirc#1{\ifmmode\setbox0=\hbox{$#1$}\dimen0=\ht0 \advance\dimen0
  by1pt\rlap{\hbox to\wd0{\hss\raise\dimen0
  \hbox{\hskip.2em$\scriptscriptstyle\circ$}\hss}}#1\else {\accent"17 #1}\fi}
  \def\ocirc#1{\ifmmode\setbox0=\hbox{$#1$}\dimen0=\ht0 \advance\dimen0
  by1pt\rlap{\hbox to\wd0{\hss\raise\dimen0
  \hbox{\hskip.2em$\scriptscriptstyle\circ$}\hss}}#1\else {\accent"17 #1}\fi}

\end{document}